\documentclass[12pt]{article}
\usepackage{amsfonts}
\usepackage{graphicx}
\usepackage{latexsym,amsmath,color}

\def\esp{\mathbb{E}}

\def\1{\mathbb{I}}

\newcounter{thm}[section]
\newcounter{appen}

\newtheorem{theor}[thm]{Theorem}
\newtheorem{cor}[thm]{Corollary}
\newtheorem{lem}[thm]{Lemma}

\newenvironment{proof}[1][Proof]{\noindent \textbf{#1.}
}{\rule{0.5em}{0.5em}}
\newtheorem{hp}{Assumption}

\begin{document}

\title{Projection-based nonparametric goodness-of-fit testing with functional covariates}
\author{Valentin Patilea\footnote{CREST (Ensai) \& IRMAR, France; patilea@ensai.fr. This author
gratefully acknowledges financial support from the Romanian National Authority for Scientific Research, CNCS-UEFISCDI, project
PN-II-ID-PCE-2011-3-0893.}\;\;\;\;\;\; C\'esar S\'anchez-Sellero\footnote{Facultad de Matem\'aticas, Universidad de Santiago de Compostela, Spain; cesar.sanchez@usc.es. This author
gratefully acknowledges support from the Spanish Ministry of Science, project MTM2008-03010, and from Ensai.}\;\;\;\;\;\; Matthieu Saumard\footnote{INSA-IRMAR, France; Matthieu.Saumard@insa-rennes.fr.}}
\date{\today}
\maketitle

\begin{abstract}

{\small This paper studies the problem of nonparametric testing for the effect of a random functional covariate on a real-valued error term. The covariate takes values in $L^2[0,1]$, the Hilbert space of the square-integrable real-valued functions on the unit interval. The error term could be directly observed as a response or \emph{estimated} from a functional parametric model, like for instance the functional linear regression. Our test is based on the remark that checking the no-effect of the functional covariate is equivalent to checking the nullity of the conditional expectation of the error  term given a sufficiently rich set of projections of the covariate. Such projections could be on elements of norm 1 from finite-dimension subspaces of $L^2[0,1]$. Next, the idea is to search a finite-dimension element of norm 1 that is, in some sense, the least favorable for the null hypothesis. Finally, it remains to perform a nonparametric  check of the nullity of the conditional expectation of the error term given the scalar product between the covariate and the selected least favorable direction. For such finite-dimension search and nonparametric check we use a kernel-based approach. As a result, our test statistic is a quadratic form based on univariate kernel smoothing and the asymptotic critical values are given by the standard normal law. The test is able to detect nonparametric alternatives, including the polynomial ones. The error term  could present heteroscedasticity of unknown form. We do no require the law of the covariate $X$ to be known. The test could be implemented quite easily and performs well in simulations and real data applications. We illustrate the performance of our test for checking the functional linear regression model.

\bigskip

\noindent Keywords: functional data regression, kernel smoothing, nonparametric testing

\bigskip \noindent MSC2000:  Primary 62G10, 62G20 ; Secondary 62G08, 62J05.}
\end{abstract}



\section{Introduction}

Consider a sample of independent copies $(U_1, X_1), \cdots, (U_n, X_n)$ of  $(U,X)$ where $U$ is a real-valued random variable and $X$
is a square-integrable random function defined on the unit interval. The problem we investigate herein is the test of the hypothesis
\begin{equation}
H_{0} :\,\, \mathbb{E} \left( U | X \right) = 0 \quad
\mbox{
\rm almost surely (a.s.)}  \label{hh0}
\end{equation}
against the nonparametric alternative $\mathbb{P} [\mathbb{E} \left( U | X \right) = 0]<1$. We consider two cases: (a) $U$ is directly observed; and (b) $U$ is not observed and is estimated as a residual of a parametric model for functional covariates and scalar responses.

There has been  substantial recent work on the theoretical study of the functional data analysis. The monographs of Ramsay and Silverman (2002, 2005)
and Ferraty (2011) provide a comprehensive landscape of the importance of the statistical methods for functional data. Estimation and prediction with functional covariates received substantial attention in the literature: for example by Ferraty and Vieu (2006), Cai and Hall (2006), Hall and Horowitz (2007), Crambes, Kneip and Sarda (2008), Yao and M\"{u}ller (2010) and the references therein.

The goodness-of-fit problem we address seems to be much less explored. There is a large literature on model checks like (\ref{hh0}) against nonparametric alternatives when $X$ takes values in a finite-dimension space, see for instance H\"{a}rdle and Mammen (1993), Stute (1997), Horowitz and Spokoiny (2001), Guerre and Lavergne (2005). In the case of functional covariate $X$, much little work was accomplished for testing against general types of alternatives. To our best knowledge, the only contribution considering the problem of testing $H_0$ against nonparametric alternatives in the cases (a) and (b) is the recent paper of Delsol, Ferraty and Vieu (2011) who extend the idea of H\"{a}rdle and Mammen (1993) to the functional covariate case. However, their results are derived under some strong assumptions, like for instance the assumptions on the rates of convergence of the so-called small ball probabilities and the law of the covariate $X$ that are supposed to be known. It is not clear how the test of Delsol, Ferraty and Vieu (2011) could be easily applied in practice, for instance for testing the goodness-of-fit of the functional linear model.
Some more substantial work was done for testing for no effect in a functional linear model, see Cardot, Ferraty, Mas and Sarda (2003), Cardot, Goia and Sarda (2007),  or for testing the functional linear model against quadratic alternatives, see Horv\`{a}th and Reeder (2011). By construction, such procedures are not able to detect general departures from the null hypothesis.

The test we introduce herein is based on a dimension reduction idea used by Lavergne and Patilea (2008) in a finite dimension setup. Our test is able to detect \emph{nonparametric} alternatives, including the polynomial ones. The variable $U$ could be heteroscedastic and we do not require the conditional variance of $U$ given $X$ to be known. We do no require the law of the covariate $X$ to be given or to be of a certain type, like for instance Gaussian. The test could be implemented quite easily and performs well in simulations and real data applications.

The paper is organized as follows. In section \ref{section1} we introduce the main notation and we derive a fundamental lemma for our approach. This lemma shows that checking condition (\ref{hh0}) is equivalent to checking the nullity of the conditional expectation of $U$ given a sufficiently rich set of projections of $X$ on elements of norm 1 from finite-dimension subspaces of $L^2[0,1]$. Next, the idea is to search in finite-dimension subspaces of $L^2[0,1]$ a least favorable  element of norm 1 and to check the nullity of the conditional expectation of $U$ given the scalar product between $X$ and the selected least favorable direction. In section \ref{sec3} we introduce the test statistic for testing of no-effect of $X$ on $U$ when $U$ is observed. Our statistic is a quadratic form, based on \emph{univariate} kernel smoothing, that behaves like a standard normal random variable under $H_0$. We prove that, under mild integrability or boundedness assumptions, the induced test is consistent against \emph{any} type of fixed alternatives and against sequences of directional alternatives approaching the null hypothesis at a suitable rate. The allowed rates are almost the same as those obtained in parametric model checks based on kernel smoothing with \emph{univariate} covariate, see for instance Guerre and Lavergne (2005) or Lavergne and Patilea (2008). In section \ref{sec44} we apply our projection-based approach for nonparametric checks of the functional regression models.
We will focus on the linear functional model, although, at the expense of longer arguments, the methodology we propose also adapts to other models, like for instance the generalized functional linear models introduced by M\"{u}ller and Stadtm\"{u}ller (2005).
In the functional regression case  the variable $U$ is the unobserved error term of the regression model and hence the test statistic is based on the estimated residuals. We still obtain standard normal critical values and consistency against nonparametric alternatives, fixed or approaching the null hypothesis. However, more restrictive conditions on the bandwidths are required due to the estimation of the slope of the functional linear model. This induces restrictions on the rate the directional alternatives may approach the null hypothesis. More difficult the estimation of the slope parameter is, slower the rate the directional alternative approach the null hypothesis should be. For estimating the slope parameter in the functional linear regression model we will focus on the standard approach based on functional principal component analysis.
In section \ref{cesar_is_the_best} an empirical study is reported. First, a wild bootstrap procedure is proposed as a means to approximate the critical values of the test statistic. Then, the results of a simulation study are briefly explained. The conclusion is that the test works well in practice. Under the null, the level is quite well respected and the power is more than acceptable even in the comparison with parametric tests. The proposed test is consistent under general alternatives. Some advices and comments are provided about the choice of the parameters involved in the new test. The test is applied to test the goodness-of-fit of the functional linear model and the functional quadratic model for the Tecator data set. Both models are rejected which indicates that more flexible models should be considered, like for instance the semiparametric index models introduced by Chen, Hall and M\"{u}ller (2011).
The proofs of our theoretical results are relegated to the appendix.

\section{Dimension reduction in nonparametric testing}\label{section1}

Let us introduce some notation. For any $p\geq 1$, let $\mathcal{S}^p =\{ \gamma\in\mathbb{R}^p: \|\gamma\|=1\}$ denote the unit hypersphere in $\mathbb{R}^p$.
Let $L^2[0,1]$ be the space of the square-integrable real-valued functions defined on the unit interval
$\langle \cdot, \cdot \rangle$ denote the inner product in $L^2[0,1]$, that is for any $X_1,X_2 \in L^2[0,1]$
$$
\langle X_1, X_2 \rangle = \int_0^1 X_1(t) X_2(t) dt.
$$
Let $\|\cdot\|_{L^2}$ be the associated norm. Hereafter  $\mathcal{R} = \{ \rho_1,\rho_2, \cdots \}$ will be an arbitrarily  fixed orthonormal basis of the function space $L^2[0,1]$, that
is $\langle  \rho_i, \rho_j \rangle  = \delta_{ij}$.
Then  the predictor process $X$  can be expanded into
\begin{equation}\label{basis_dec}
X(t) = \sum_{j=1}^\infty x_j \rho_j(t),
\end{equation}
where the random coefficients $x_j$ are given by $x_j = \langle  X, \rho_j \rangle $. For a fixed positive integer $p$, $X^{(p)}\in L^2[0,1] $ will be the projection of $X$ on the subspace generated by the first $p$ elements of the basis  $\mathcal{R}$, that is
$$
X^{(p)}(t) = \sum_{j=1}^p x_j \rho_j(t).
$$
Let us notice that $\|X^{(p)}\|_{L^2}$ coincides with the Euclidean norm of the vector $(x_1, \cdots,x_p)$ in $\mathbb{R}^p$. By abuse we also identify $X^{(p)}$ with the $p-$dimension random vector $(x_1, \cdots,x_p).$ On the other hand, for any integer $p>1$ and non random vector $\gamma = (\gamma_1, \cdots,\gamma_p) \in\mathbb{R}^p$, we consider by abuse $\gamma$ an element in $L^2[0,1]$
with $(\gamma_1, \cdots,\gamma_p, 0,0, \cdots )$  the coefficients of its expansion and  hence $\langle X, \gamma \rangle = \langle X^{(p)}, \gamma \rangle = \sum_{i=1}^p x_j \gamma_j$. In the following we will also use $\beta = \sum_{j=1}^\infty b_j \rho_j(t)$ to denote a non random element of $L^2[0,1]$.


Our approach relies on the following lemma, an extension of Lemma 2.1 of Lavergne and Patilea (2008) and Theorem 1 in Bierens (1990) to Hilbert space-valued conditioning random variables.  The result shows that for checking nullity of a conditional expectation, it is equivalent to consider
expectations conditional on $X$ and expectations conditional on $L^2 [0,1]$ projections of $X$ on a sufficiently rich set of directions.

\begin{lem}
\label{lem1} Let $X\in L^2[0,1]$ and $Z\in \mathbb{R}$ be random
variables. Assume that  $\mathbb{E} | Z | <\infty $ and $\mathbb{E} ( Z ) =0.$

(A) The following statements are equivalent:
\begin{enumerate}
\item $\mathbb{E}(Z\mid X)=0$ a.s.
\item $\mathbb{E}(Z\mid \langle X,\beta \rangle )=0$ a.s. $
 \forall \beta \in L^2[0,1]$ with $\|\beta\|_{L^2} = 1$.
\item for any integer $p\geq 1$, $\mathbb{E}(Z\mid \langle X, \gamma \rangle )=0$ a.s. $
 \forall \gamma \in\mathcal{S}^p.$
\item\label{point4} for any integer $p\geq 1$, $\mathbb{E}(Z\mid  X^{(p)} )=0$ a.s.
\end{enumerate}

(B) Suppose in addition that for any positive real number $s$,
\begin{equation}\label{cond_x_cond}
\mathbb{E}(|Z|\exp\{ s \|X\| \}) <\infty.
\end{equation}
If $\mathbb{P} [\mathbb{E} (Z\mid X) = 0] <1$, then there exists a positive integer $p_0\geq 1$ such that for any  integer $p> p_0$, the set
$$
\{\gamma\in\mathcal{S}^p : \mathbb{E}(Z \mid \langle X, \gamma \rangle )=0 \,\, a.s.\, \}
$$
has Lebesgue measure zero on the unit hypersphere  $\mathcal{S}^p$ and is not dense.
\end{lem}

\bigskip

Point (A) is a cornerstone for proving the behavior of our test under the null and the alternative hypothesis. Point (B) shows that in applications it will not be difficult to find directions $\gamma$ able to reveal the failure of the null hypothesis (\ref{hh0}). Under the additional assumption (\ref{cond_x_cond}) such directions represent almost all the points on the unit hyperspheres $\mathcal{S}^p$, provided $p$ is sufficiently large. The assumption (\ref{cond_x_cond}) is not restrictive for testing purposes. Indeed, if $X$ does not satisfy condition (\ref{cond_x_cond}), it suffices to transform  $X$ into some variable  $W\in L^2[0,1]$ such that the $\sigma-$field generated by $W$ is the same as the one generated by $X$ and the variable $W$ satisfies condition (\ref{cond_x_cond}).\footnote{For instance, given $X = \sum_{j\geq 1} x_j \rho_j$, one may build $w_j = a_j \arctan(x_j)$, where $a_j$ are non random such that $\sum_{j\geq 1} a_j^2 < \infty$ and may use the bounded random function  $W = \sum_{j\geq 1} w_j \rho_j\in L^2[0,1]$ (bounded means $\|W\|$ is a bounded random variable) instead of $X$ in the conditioning.} Clearly, when $U$ is the error term in some functional regression model for which one wants to check the goodness-of-fit, one should use a transformation of $X$ only \emph{after} estimating the errors in the model.

The following new formulations of $H_{0}$ are direct consequences of Lemma \ref{lem1}-(A).

\begin{cor}
\label{charac} Consider a real-valued random variable  $U$
such that  $\mathbb{E}| U| <\infty $. Let $\omega(\beta,t)$, $\beta\in L^2[0,1]$ and $t\in\mathbb{R}$, be a real-valued function such that $\omega(\beta,\langle X,\beta\rangle)>0$ for all $\|\beta\|_{L^2} =1$. For any $p\geq 1$, let $w_p(\gamma,t)$, $\gamma\in\mathbb{R}^p$ and $t\in\mathbb{R}$, be a real-valued function such that $w_p(\gamma,\langle X,\gamma\rangle)>0$ for all $\|\gamma\| =1$.
The following statements are equivalent:
\begin{enumerate}
\item The null hypothesis (\ref{hh0}) holds true.
\item
\begin{equation}
\max_{\beta\in L^2[0,1],\; \|\beta\|_{L^2} = 1} \mathbb{E} \left[ U \mathbb{E}
\left( U  |\langle X,\beta\rangle \right) \omega (\beta, \langle X,\beta\rangle ) %
\right] = 0  .  \label{nh0_a}
\end{equation}
\item\label{point_33}  for any $p\geq 1$ and any set $B_p\subset\mathcal{S}^p$ with strictly positive Lebesgue measure on the unit hypersphere  $\mathcal{S}^p,$
\begin{equation}
\max_{\gamma\in B_p} \mathbb{E} \left[ U \mathbb{E}
\left( U  |\langle X,\gamma\rangle\right) w_p (\gamma, \langle X,\gamma\rangle)
\right] = 0.  \label{nh0_b}
\end{equation}\end{enumerate}
\end{cor}

\section{Testing the effect of a functional covariate}

\label{sec3}

\setcounter{equation}{0}

We introduce a general approach for nonparametric testing of the effect of a functional
covariate $X$ on a real-valued random variable $U$. For simplicity, here  we assume that $\mathbb{E}(U)=0$, the nonzero mean case is contained in the setup considered in section \ref{sec44} below. Our approach is based on Corollary \ref{charac}-(\ref{point_33}) and \emph{univariate} kernel smoothing. In this way we avoid the problem of smoothing in infinite-dimension, in particular we avoid using the small ball function required in the kernel regression with functional covariates, see Ferraty and Vieu (2006), Delsol, Ferraty and Vieu (2011).

To avoid handling denominators close to zero, we set the weight function $\omega (\gamma,\cdot )$
in Corollary  \ref{charac} equal to the density of $\langle X, \gamma\rangle $, denoted by $f_{\gamma }(\cdot )$,
which is assumed to exist for any $\gamma $. For any $\gamma\in \mathbb{R}^p$, let
\begin{equation*}
Q(\gamma )=\mathbb{E}\{U\; \mathbb{E}[U\mid \langle X,
\gamma\rangle ]f_{\gamma }(\langle X,\gamma\rangle )\}=\mathbb{E}\{\mathbb{E}^{2}[U
\mid \langle X,\gamma\rangle ]f_{\gamma }(\langle X,\gamma\rangle )\}.
\end{equation*}
For any $p \geq 1$, let $B_p\subset\mathcal{S}^p$ be a set with strictly positive Lebesgue measure in $\mathcal{S}^p.$ By Corollary \ref{charac}, the null hypothesis (\ref{hh0}) holds true if and only if
\begin{equation}\label{test_a}
\forall p\geq 1,\quad \max_{\gamma\in B_p}Q(\gamma )=0.
\end{equation}

\subsection{The test statistic}\label{test_stat_sec}

In view of equation (\ref{test_a}), our goal is to estimate $Q(\gamma).$ With at hand a  sample  of $(U,X)$, define
\begin{equation*}
Q_{n}\left(\gamma \right) =\frac{1}{n(n-1)}\sum\limits_{1\leq i\neq
j\leq n}U_{i} U_{j} \frac{1}{h}
K_{h}\left( \langle X_{i}-X_{j},\gamma\rangle  \right),\quad \gamma\in\mathcal{S}^p,
\end{equation*}%
where  $K_{h}\left( \cdot \right) =K\left( \cdot /h\right) $, where $K(\cdot )$
is a kernel and $h$ a bandwidth. In the case of finite dimension covariates, the function $\gamma\mapsto Q_{n}(\gamma)$ is the statistic considered by Lavergne and Patilea (2008), see also Bierens (1990). For  fixed $p$ and $\gamma\in\mathcal{S}^p$, it is well-known that $Q_{n}\left(\gamma \right)$ has
asymptotic centered normal distribution with rate $nh^{1/2}$ under $H_{0}$; see for instance  Guerre and Lavergne (2005).
We will show that the asymptotic normal distribution is preserved even when $p$ grows at a suitable rate with the sample size.
On the other hand, Lemma \ref{lem1}-(B) indicates that if $p$ is large enough, the maximum of $Q\left(\gamma\right)$ over $\gamma$  stays away from zero when $H_{0}$ fails.

For a fixed $p$ the statistic $Q_{n}(\gamma)$ is expected to be close to $Q(\gamma  )$ uniformly in $\gamma$. Then a natural idea would be to build a test statistic using the maximum of $Q_{n}(\gamma)$ with respect to $\gamma$. However, there is an additional difficulty one faces in the functional data framework since then one  has to let $p$ to grow to infinity with the sample size, and hence the closeness between $Q_{n}(\gamma)$ and  $Q(\gamma)$ requires a more careful investigation. On the other hand, like in the finite dimension covariate case, under $H_0$ one expects $Q_{n}(\gamma)$ to converges to zero for any $p$ and $\gamma$ and thus the objective function of the maximization problem to be flat.

We will choose a direction $\gamma$ as the least favorable direction for the null hypothesis $H_0$ obtained from a penalized criterion based on a standardized version of $Q_{n}\left(\gamma \right)$. Lavergne and Patilea (2008) and Bierens (1990) considered this idea using  $Q_{n}\left(\gamma \right)$. Here we use a standardized version of $Q_{n}\left(\gamma \right)$. More precisely, fix some $\beta_0\in L^2[0,1]$ that could be interpreted as an initial \emph{guess} of an unfavorable direction for $H_0.$
Let $b_{0j},$ $j\geq 1$,  be the coefficients in the expansion of $\beta_0$ in the basis $\mathcal{R}$. For any given $p\geq 1$ such that $\sum_{j=1}^p b_{0j}^2 >0$, let
$$
\gamma_0^{(p)} = (b_{01},\cdots,b_{0p})/\sqrt{\sum_{j=1}^p b_{0j}^2}\; \;.
$$
Let $\widehat{v}_{n}^{2} (\cdot)$ be an estimate of the
variance of $nh^{1/2}Q_{n}(\cdot) $. Given  $B_p\subset \mathcal{S}^p$ with strictly positive Lebesgue measure in $\mathcal{S}^p$ that contains $\gamma_0^{(p)}$,  the least favorable direction $\gamma$ for $H_0$ is defined as
\begin{equation}  \label{bet}
\widehat{\gamma }_{n} = \arg \max_{\gamma\in B_p} \left[ n h
^{1/2} Q_{n}(\gamma)/\widehat{v}_{n}(\gamma) - \alpha _{n} \mathbb{I}_{\left\{
\gamma \neq \gamma_0^{(p)}  \right\}} \right] \;,
\end{equation}
where $\mathbb{I}_A$ is the indicator function of a set $A$,  and $\alpha_n$, $n\geq 1$ is a sequence of positive real numbers increasing to infinity at an appropriate rate that depends on the sample size and the rates of $h$ and $p$ and that will be made explicit below. Let us notice that the maximization used to define $\widehat{\gamma }_{n}\in\mathcal{S}^p$ is a finite dimension optimization problem. The choice of $\beta_0$, and thus of $\gamma_0^{(p)}$, is theoretically irrelevant, it does not affect the asymptotic critical values and the consistency results. However, in practice the choice of $\beta_0$ could be related to a priori information of the practitioner on a class of alternatives, like for instance the class of functions depending only on $\langle X,\beta_0\rangle$. The empirical investigation we report in section \ref{cesar_is_the_best} suggests that working with a standardized version of $Q_{n}\left(\gamma \right)$ simplifies the choice of $\alpha_n$ in applications.

We will prove that with  suitable rates of increase for $\alpha_n$ and $p$ and decrease for $h$, the probability of the event $\{
\widehat{\gamma}_{n} = \gamma_{0}^{(p)}\}$ tends to 1 under $H_{0}$.
Hence $Q_{n}(\widehat{\gamma}_{n})/\widehat{v}_{n}(\widehat{\gamma})$ behaves asymptotically like $
Q_{n}(\gamma_{0}^{(p)})/\widehat{v}_{n}(\gamma_0^{(p)})$, even when $p$ grows with the sample size.
Therefore the test statistic we consider is
\begin{equation}\label{test_stat}
T_{n} = n h^{1/2} \frac{Q_{n}(\widehat{\gamma}_{n})}{
\widehat{ v}_{n} (\widehat{\gamma}_{n})} \; .
\end{equation}
We will show that an asymptotic $\alpha$-level test is given by $\mathbb{I} \left(
T_{n}\geq z_{1- \alpha} \right) $, where $z_{1-\alpha}$ is the
$(1-\alpha)$-th quantile of the standard normal distribution.

\subsection{Estimating the variance}\label{est_var_sec}

To find the direction $\widehat\gamma_n$ and to build the test statistics (\ref{test_stat}), we need to estimate in some way the variance of $ nh^{1/2}Q_{n}( \gamma) $. The approach that is expected not to inflate the variance estimate under the alternatives and thus to guarantee better power small finite samples would involve the estimations of the conditional variance of $ nh^{1/2}Q_{n}(\gamma) $ given $X_i$'s
which writes
\begin{equation}\label{th_var_q}
\tau_{n}^{2}\left( \gamma \right) =\frac{2}{n(n-1)h}\sum\limits_{j\neq i}
\sigma_p^{2}(X_{i}^{(p)})\sigma_p ^{2}(X_{j}^{(p)})  K_{h}^{2}\left(
\langle X_{i}-X_{j},\gamma \rangle  \right) \; ,  
\end{equation}
where $\sigma_p^{2}(X^{(p)}) = Var[U\mid X^{(p)}]$. An estimator can  be easily obtained by replacing $\sigma^2_p(\cdot)$ with an estimate in the last expression. In theory, a good solution would be to use a nonparametric estimate of the $p-$variate function $\sigma_p ^{2}(\cdot)$, but this is practically infeasible given that it expected to let $p$ to grow with the sample size. A simple and convenient solution with high-dimension covariates is then
\begin{equation}
\widehat{\tau}_{n}^{2}\left( \gamma \right) =\frac{2}{n(n-1)h}\sum\limits_{j\neq
i}U_{i}^2U_{j}^2
K_{h}^{2}\left(\langle X_{i}-X_{j},\gamma \rangle \right).
\label{var_est_tau}
\end{equation}
Since, under the null hypothesis $\widehat \gamma_n =\gamma_0^{(p)}$ with probability tending to 1, a first variance estimator we propose is
\begin{equation}\label{var_est1a}
\widehat{v}_{n}^2 (\widehat\gamma_n)= \widehat{v}_{n}^2 (\widehat\gamma_n,\gamma^{(p)}_0) =  \min \left(\widehat{\tau}_{n}^2 (\widehat \gamma_n), \widehat{\tau}_n^2 (\gamma^{(p)}_0)
\right).
\end{equation}

On the other hand, let us notice that $\tau_{n}^{2}( \gamma_0^{(p)}) - \mathbb{E}[\tau_{n}^{2}( \gamma_0^{(p)} )]$ is expected to converge to zero. Moreover, under the null hypothesis
\begin{multline*}
\mathbb{E}[\tau_{n}^{2}( \gamma_0^{(p)} )] = \mathbb{E}[\mathbb{E}[\tau_{n}^{2}( \gamma_0^{(p)} ) \mid \langle X_{1}, \gamma_0^{(p)} \rangle,\cdots, \langle X_{n}, \gamma_0^{(p)} \rangle]]\\
= \mathbb{E}\left\{ \frac{2}{n(n-1)h}\sum\limits_{j\neq i}
Var[U_i \mid \langle X_{i}, \gamma_0^{(p)} \rangle]\; Var[U_j \mid \langle X_{j}, \gamma_0^{(p)} \rangle]  K_{h}^{2}\left(
\langle X_{i}-X_{j},\gamma_0^{(p)} \rangle  \right) \right\},
\end{multline*}
and $0< \underline{\sigma}^2 \leq Var[U \mid \langle X, \gamma_0^{(p)} \rangle]\leq\overline{\sigma}^2<\infty$.  Next, notice that the conditional variance of $U$ given $\langle X, \beta_0 \rangle$ is the same under $H_0$ and under any alternative that depends only on $\langle X, \beta_0 \rangle$. Finally, notice that in any case $\mathbb{E}(U^2)\geq \mathbb{E}[Var(U\mid \langle X, \beta_0 \rangle)]$. All these facts suggest that a compromise for estimating the  variance of $ nh^{1/2}Q_{n}(
\beta_0 ^{(p)}) $ would be
\begin{equation}\label{var_est1b}
\widehat{v}_{n}^2=\widehat{v}_{n}^2 (\gamma_0^{(p)}) = \frac{2}{n(n\!-\!1)h}\!\sum\limits_{j\neq i}
\widehat\sigma^{2}_{\gamma_0^{(p)}}( \langle X_{j}, \gamma_0^{(p)} \rangle)\ \widehat\sigma^{2}_{\gamma_0^{(p)}}( \langle X_{j}, \gamma_0^{(p)} \rangle)  K_{h}^{2}\left(\!
\langle X_{i}-X_{j},  \gamma_0^{(p)} \rangle\!  \right) ,
\end{equation}
where $\widehat\sigma_{\gamma_0^{(p)}}^2 (\cdot)$ is some nonparametric estimate of the univariate function $\sigma_{\gamma_0^{(p)}}^2(t)= Var(U\mid \langle X, \gamma_0^{(p)} \rangle = t)$ satisfying the condition
\begin{equation}
\sup_{1\leq i \leq n}\left\vert \frac{\widehat{\sigma }^{2}_{\gamma_0^{(p)}}(\langle X_{i}, \gamma_0^{(p)} \rangle)}{\sigma^{2}_{\gamma_0^{(p)}}(\langle X_{i}, \gamma_0^{(p)} \rangle)} -1\right\vert =o_{\mathbb{P}}(1) \; ,  \label{nonp}
\end{equation}
Different nonparametric estimators can be used, for instance a kernel estimator like in Lavergne and Patilea (2008).
We will prove below that both variance estimators (\ref{var_est1a}) and (\ref{var_est1b}) guarantee the standard normal asymptotic critical values and consistency of our test. In simulations, better power under the alternative was obtained when using the variance estimator (\ref{var_est1b}). The drawback of this estimator is the computational cost and the choice of an additional bandwidth for the estimate $\widehat\sigma_{\gamma_0^{(p)}}^2 (\cdot)$. However, common choices of this bandwidth work well in practice.


\subsection{Behavior under the null hypothesis}\label{beh_null_hyp}

Let us introduce a first set of assumptions. Below $0_p\in\mathbb{R}^p$
denotes the null vector of dimension $p$. Moreover, $\mathcal{F}[\cdot]$ denotes the Fourier transform, cf. Rudin (1987).

\setcounter{hp}{3}

\begin{hp}
\label{D}
\hspace{-0.1cm}
\begin{enumerate}
\item[(a)] The random vectors $(U_{1},X_{1})
,\ldots ,(U_{n},X_{n})$ are independent draws from the random vector $(U,X)\in\mathbb{R} \times L^2[0,1]$ that satisfies  
$\mathbb{E}|U|^{m}<\infty $ for some $m > 11$.

\item[(b)]
$\exists \ \underline{\sigma}^{2}$ and $\overline{\sigma}^{2}$
such that  $0 < \underline{\sigma}^{2} \leq Var(U\mid X) \leq
\overline{\sigma}^{2} < \infty$ almost surely.

\item[(c)] The sets $B_p\subset\mathcal{S}^p$, $p\geq 1$ appearing in (\ref{bet}) are such that:
\begin{enumerate}
\item[(i)]
there exist constants $C_1,\delta >0$ (independent of $n$ and $p$) such that
$\forall p\geq 1$ and $\forall \gamma\in B_p$, the  variable $\langle X, \gamma \rangle$ admits a density
$f_{\gamma}(\cdot)$ and $$C_1^{-1}\leq  \int_{\mathbb{R}} \{f_{\gamma}^2 \mathbb{I}(f\leq 1) + f_{\gamma}^{2+\delta} \mathbb{I}(f > 1) \}\leq  C_1;$$

\item[(ii)] there exists $C_2,\epsilon>0$ such that
$\int_{|x|\leq \epsilon} |\mathcal{F}[f_\gamma] |^2(x)dx \geq C_2$, $\forall p\geq 1,$ $\forall \gamma\in B_p$;

\item[(iii)] the initial `guess' $\beta_0$ satisfies the condition: $\exists C_3$ such that $f_{\gamma_0^{(p)}}\leq C_3$, $\forall p\geq 1.$

\item[(iv)]
$B_p\times 0_{p^{\,\prime} - p} \subset B_{p^{\,\prime}},$ $\forall 1\leq p< p^{\,\prime}$.
\end{enumerate}
\label{ass_dc}
\end{enumerate}
\end{hp}

\setcounter{hp}{10}
\begin{hp}
\label{K} \hspace{-0.1cm}
\begin{enumerate}
\item[(a)] The kernel $K$ is a continuous  density of bounded variation with strictly positive Fourier transform on the real line.
\item[(b)] $h\rightarrow 0$ and $\left(nh^{2}\right)^{\alpha} / \ln n \rightarrow
\infty$ for some $\alpha \in (0,1)$.

\item[(c)] $p\geq 1$ increases to infinity with $n$ and there exists a constant $\lambda>0$ such that $p\ln^{-\lambda} n $ is  bounded.
\end{enumerate}
\end{hp}

Let us comment on these assumptions. The bounded variation of $K$, in particular this means $K$ is bounded, is a very mild condition that allows to easily bound covering numbers of families of functions indexed by $\gamma$. Continuity and bounded variation guarantee that $K$ can be recovered by  inverse Fourier transform. The role of technical assumption of positive Fourier, that is satisfied by triangular, normal, logistic, Student, or Laplace densities, will be explained below. In Assumption \ref{K}-(c), it is also possible to let $p$ to grow with the sample size at a polynomial rate, instead of the logarithmic rate. However, we will see below that, in theory, this could induce a loss of power for our test. There is a trade off between the moment conditions one imposes for $U$ and the range of rates allowed for the bandwidth and the growth rate for $p$~: higher moments will be needed for wider ranges and faster rates for $p$. For bandwidths and $p$ satisfying Assumption \ref{K}-(b,c) it suffices to take $m > 11$ in Assumption \ref{D}-(a); see the proof of Lemma \ref{leem1}. Let us notice that Assumption \ref{ass_dc}-(b) implies that $\forall p\geq 1$,  $0 < \underline{\sigma}^{2} \leq \mathbb{E}(U^{2}\mid X^{(p)})\leq \overline{\sigma}^{2} < \infty$ almost surely.
Finally, let us comment on Assumption \ref{ass_dc}-(c). On one hand, a key issue in the proof of Lemma \ref{leem1} below and some of the subsequent proofs will be to control the rate of $\mathbb{E}[h^{-1}K_h(\langle X_1 - X_2,\gamma \rangle)]$ uniformly in $\gamma\in B_p$ as $p$ grows and $h$ decreases with the sample size. To reduce technicalities, we choose the convenient solution that consists in trying to   bound this quantity by a constant. Using the Fourier transform and Plancherel theorem, this is guaranteed by a condition like $\int_{\mathbb{R}} f_{\gamma}^2 \leq  C_1$, $\forall \gamma\in B_p$.
In the proofs for the functional linear model we have to strengthen this condition and add $\int_{\mathbb{R}} f_{\gamma}^{2+\delta} \mathbb{I}(f > 1) \leq  C_1,$ $\forall \gamma\in B_p$, for some arbitrary small $\delta>0.$
Such sufficient conditions could be easily achieved for instance if the coefficients $x_j$ of the expansion of $X$ are independent. Then it suffices to fix some $k\geq 1$ such that the density of $x_k$ is bounded and some small $c$ independent of $p$  and to take $B_p = \{(\gamma_1,\cdots,\gamma_k,\cdots,\gamma_p) \in\mathcal{S}^p : |\gamma_k|\geq c\}$. This simple idea could be useful in many other cases than the one of independent coefficients $x_j$. On the other hand, we have to keep the variance estimate in the denominator of the test statistic (\ref{test_stat}) away from zero. For this we have to ensure that  $\mathbb{E}[h^{-1}K_h^2(\langle X_1 - X_2,\gamma \rangle)]$ is bounded away from zero uniformly in $\gamma\in B_p$ as $p$ grows and $h$ decreases with the sample size. One easy way to ensure this is to use again the Fourier transform properties, the positiveness of  $\mathcal{F}[K]$ and to impose the positive uniform lower bound for the integral of square of $\mathcal{F}[f_\gamma]$ in a neighborhood of the origin, which necessarily induces a uniform lower bound for $\int_{\mathbb{R}} f_{\gamma}^2.$
Assumptions \ref{ass_dc}-(c)(iii) will complete the sufficient conditions for deriving standard normal critical values using  the central limit theorem for $U-$statistics of  Guerre and Lavergne (2005, Lemma 2).
To summarize, the choice of $\beta_0$ and $B_p$ will be decided in the applications and will also depend on the law of $X$ and the choice of the basis $\mathcal{R}$. In view of our extensive simulation experiment, we argue that the choice of $B_p$ is not an issue in applications, one can confidently perform the optimization on the whole hypersphere $\mathcal{S}^p$.
Finally, the condition $B_p\times 0_{p^{\,\prime} - p} \subset B_{p^{\,\prime}},$  $\forall p< p^{\,\prime},$ is a mild technical condition that combined with Lemma \ref{lem1}-(A) greatly simplifies the proof of the consistency of our test.

The first step is the study of the behavior of the process $Q_{n}(\gamma),$ $\gamma\in B_p$, under $H_{0}$ when $p$ is allowed to increase with the sample size.

\begin{lem}
\label{leem1}  Under  Assumptions \ref{D} and \ref{K} and if $H_0$ holds true,
$$\sup_{\gamma \in B_p\subset\mathcal{S}^p}|Q_{n}(\gamma)|= O_{\mathbb{P}
}(n^{-1}h^{-1/2}p^{3/2} \ln n) .$$ Moreover, if $\widehat \tau_n^2 (\gamma)$ is the estimate defined in equation (\ref{var_est_tau}),
\begin{equation*}
\sup_{\gamma \in B_p\subset\mathcal{S}^p} \{1/\widehat \tau_n^2 (\gamma)\} = O_{\mathbb{P}}(1).
\end{equation*}
If in addition condition (\ref{nonp}) holds true,
$1/ \widehat v_n^2 = O_{\mathbb{P}}(1)$ with $\widehat v_n^2 $  defined in (\ref{var_est1b}).
\end{lem}

We now describe the behavior of $\widehat{\gamma }_{n}$ under $H_{0}$. A suitable rate $\alpha_n$ will make $\widehat{ \gamma}_{n}$ to be equal to $\gamma_{0}^{(p)}$ with high probability. Under the null, $\alpha_n$ has to grow to infinity sufficiently fast to render the probability of the event $\{ \widehat{ \gamma}_{n} = \gamma_{0}^{(p)} \}$ close to 1. We will see below that, for better detection of   alternative hypothesis,  $\alpha_n$ should grow as slow as possible. Indeed, slower rates for $\alpha_n$ will allow the selection of  directions $\hat\gamma_n$  that could be better suited than $\gamma_0^{(p)}$ for revealing the departure from the null hypothesis. The rate of $p$ is also involved in the search of a trade-off for the rate of $\alpha_n$: larger $p$ renders slower the rate of uniform convergence to zero of $Q_n (\gamma)$, $\gamma\in B_p$, and hence requires larger $\alpha_n$.

\begin{lem}
\label{beta} Under Assumptions \ref{D}, \ref{K}, and condition (\ref{nonp}) if the variance estimator is the one defined in (\ref{var_est1b}), for a positive sequence
$\alpha _{n} $, $n\geq 1$ such that $\alpha _{n}/\{ p^{3/2} \ln n \} \rightarrow
\infty$,  $$\mathbb{P}(\widehat{ \gamma}_{n} = \gamma_{0}^{(p)})
\rightarrow 1, \quad \text{ under } H_{0}.$$
\end{lem}

The following result shows that the asymptotic critical values of our test statistic are standard normal.

\begin{theor}
\label{as_law} Under the conditions of Lemma \ref{beta} and if the hypothesis $H_0$ in (\ref{hh0}) holds true, the  test statistic $
T_{n} $ converges in law to a standard normal. Consequently, the test given by $\mathbb{I}(T_n \geq z_{1- a})$, with $z_a$  the $(1- a)-$quantile of the standard normal distribution,  has asymptotic level  $ a.$
\end{theor}


\subsection{The behavior under the alternatives}

First let us give an intuition on the reason why our test is consistent. Consider the alternative hypothesis $$H_1:\; \mathbb{P}[\mathbb{E}(U\mid X)=0] <1.$$  The way the statistic $T_n$ is constructed guarantees the consistency of our test against  $H_1.$ Indeed, we can write
\begin{eqnarray}
T_{n} &=&\frac{nh^{1/2}Q_{n}(
\widehat{\gamma }_{n})}{ \widehat{v}_{n}(\widehat{\gamma }_{n}) }  \notag \\
& = & \max_{\gamma\in B_p}\left\{ nh^{1/2} Q_{n}(\gamma )/ \widehat{v}_{n}(\gamma )
-\alpha _{n} \mathbb{I}_{\{
\gamma \neq \gamma_{0}^{(p)}\}} \right\} +\alpha _{n} \mathbb{I} _{\{
\widehat\gamma_n \neq \gamma_{0}^{(p)}\}}   \notag \\
&\geq &  \frac{\max_{\gamma\in B_p} nh^{1/2} Q_{n}(\gamma)}{\widehat{v}_{n}(\gamma_{0}^{(p)}) } -\alpha _{n}
\geq \frac{nh^{1/2} Q_{n}(\gamma ) }{\widehat{v}_{n}(\gamma_{0}^{(p)})} -\alpha_{n}, \quad \forall \gamma\in B_p\subset\mathcal{S},
 \label{eqaa}
\end{eqnarray}
with $\widehat{v}_{n}(\gamma_{0}^{(p)})$ equal to  $\widehat{\tau}_{n}(\gamma_{0}^{(p)})$ defined in  (\ref{var_est1a}) or equal to $\widehat{v}_{n}$ defined in (\ref{var_est1b}). Since $\mathbb{E}(U^2\mid X) \geq \underline{\sigma}^2 $ and $Var(U\mid \langle X,\gamma_0^{(p)}\rangle) \geq \underline{\sigma}^2 $, it is clear that
$1/\widehat{v}_{n}(\gamma_{0}^{(p)})=O_{\mathbb{P}}(1)$ for both variance estimates introduced above. On the other hand, from Lemma \ref{lem1}, there exists $p_0$ and $\widetilde \gamma\in B_{p_0}$ such that the expectation of $ Q_{n}(\widetilde \gamma )$ stays away from zero as the sample size grows to infinity and $h$ decrease to zero. On the other hand, for any $p>p_0$ and any $n$ and $h$, clearly $\max_{\gamma \in B_p} Q_{n}(\gamma ) \geq Q_{n}(\widetilde \gamma )$, because $B_{p_0} \times 0_{p-p_0} \subset B_{p}$. All these facts show why our test is omnibus, that is consistent against nonparametric alternatives, provided that $p\rightarrow \infty.$

To formalize the consistency result, let us fix some real-valued function $\delta(X)$ such that $\mathbb{E}[\delta(X)] = 0$ and $0<\mathbb{E}[\delta^4(X)]<\infty$,  and some sequence of real numbers $r_n$ that could decrease to zero (the case $r_n\equiv 1$ is also included). Consider the sequence of alternatives
\begin{equation}\label{pitman_alt}
H_{1n}:\; U = U^0 + r_n \delta(X),\quad  n\geq 1, \quad\text{with} \;\; \mathbb{E}(U^0\mid X) = 0.
\end{equation}
We show below that such directional alternatives can be detected as soon as $r_n^2 n h^{1/2}  / \alpha_n $ tends to infinity. This is exactly the same condition as in Lavergne and Patilea (2008). However, in the functional data framework, to obtain the convenient standard normal critical values,  we need   $1/\alpha_n =o(p^{-3/2}\ln^{-1} n)$. Hence, the rate $r_n$ at which the alternatives $H_{1n}$ tend to the null hypothesis should satisfy  $r_n^{2}n h^{1/2}/\{p^{3/2}\ln n\}\rightarrow \infty $.

\begin{theor}
\label{altern}
Suppose that
\begin{enumerate}
\item[(a)] Assumption \ref{D}  holds true with $U$ replaced by $U^0$;
 \item[(b)] Assumption \ref{K} is satisfied, and so is the condition (\ref{nonp}) if the variance estimator is the one defined in (\ref{var_est1b});
  \item[(c)] $\alpha _{n}/\{ p^{3/2} \ln n \} \rightarrow
\infty$ and
$r_n $, $n\geq 1$ is such that $r_n^2 n h^{1/2}  / \alpha_n\rightarrow \infty$;
  \item[(d)] $\mathbb{E}[\delta(X)] = 0$ and $0<\mathbb{E}[\delta^4(X)]<\infty$.
\end{enumerate}
Then the test based on $T_n$ is consistent against the sequence of alternatives $H_{1n}$  if there exists $ p\geq 1$ and  $\widetilde \gamma\in B_p$ such that $\mathbb{P}(\mathbb{E}[\delta(X)\mid \langle X,\widetilde \gamma\rangle]=0)<1$ and one of the following conditions is satisfied:
\begin{enumerate}
\item the density $f_{\widetilde \gamma}$ is bounded;
\item the function $\mathbb{E}[\delta(X)\mid \langle X,\widetilde \gamma\rangle = \cdot ]f_{\widetilde\gamma} (\cdot)$ is bounded;
\item the Fourier transform  of $\mathbb{E}[\delta(X)\mid \langle X,\widetilde \gamma\rangle = \cdot ]f_{\widetilde\gamma} (\cdot)$ is integrable on $\mathbb{R}$.
\end{enumerate}

\end{theor}

\bigskip

Let us recall that the existence of $p$ and $\widetilde \gamma\in B_p$ such that $\mathbb{P}[\delta(X)\mid \langle X,\widetilde \gamma\rangle]=0)<1$ is guaranteed by Lemma \ref{lem1}.

\section{Testing the goodness-of-fit of parametric models}
\label{sec44}

Here we apply our projection-based testing methodology for testing the goodness-of-fit of the functional linear regression model against \emph{nonparametric} alternatives satisfying mild technical conditions. Hence we provide a simple goodness-of-fit procedure for a widely used model. To the best of our knowledge, our results are completely new in the functional regression framework.

Let $U$ be a real-valued random variable and  $X$ be a random variable with values in $L^2[0,1]$. The model we want to test is the functional linear model defined by
\begin{equation*}
Y=a + \langle b,X\rangle + U,
\end{equation*}
where $b\in L^2[0,1]$ and $a\in\mathbb{R}$ are unknown parameters.  The null hypothesis is
\begin{equation}
H_{0} :\,\, \mathbb{E} \left( U | X \right) = 0 \quad
\mbox{
\rm a.s.}  \label{hh0bis}
\end{equation}
In order to formally establish consistency against nonparametric alternatives, we will consider a sequence  of local alternatives like in (\ref{pitman_alt}).

Like in Assumption \ref{D}  we consider that $(U_1, X_1), \cdots, (U_n, X_n)$ are independent copies of  $(U,X),$
but now the observations are $(Y_1, X_1), \cdots, (Y_n, X_n)$. Hence the error term $U$  has to be estimated in a preliminary step from the estimates of the parameters $a$ and $b$. We will investigate the behavior of our test statistic under the null and under the alternatives for a
generic estimate of the slope with suitable rate of convergence. Next, we will get into the details in the standard case of the slope estimate based on the functional principal component analysis. In particular, we will see  how the difficulty of estimating the parameters in the functional regression model could perturb the properties of test. To keep the technical conditions readable, hereafter we will assume that $\mathbb{E}(|U|^m)<\infty$ for any $m\geq 1.$


\subsection{The test statistic and the behavior under the null hypothesis}
The test statistic is  similar to the one we proposed for testing the effect of a functional covariate.
Let $\beta_0$, $\gamma^{(p)}_0$, $\mathcal{S}^p$ and $B_p$ be defined as in section \ref{sec3}. Let $\widehat b\in L^2[0,1]$ denote a generic  estimator of the slope $b$ and let
$$
\widehat a = \overline{Y}_n - \int_0^1 \widehat{b}(t)\overline{X}_n(t)dt = a - \int_0^1\{\widehat{b}(t)-b(t)\}\overline{X}_n(t) dt + \overline{U}_n,
$$
where  $\overline{U}_n=n^{-1}\sum_{i=1}^n  U_i.$
 Let $\widehat{U}_i=Y_i-\widehat{a}- \langle \widehat{b},X_i \rangle$ be the residuals and let
\begin{equation*}
Q_{n}(\gamma; \widehat a,\widehat b ) =\frac{1}{n(n-1)}\sum\limits_{1\leq i\neq
j\leq n}\widehat{U}_{i} \widehat{U}_{j} \frac{1}{h}
K_{h}\left( \langle X_{i}-X_{j},\gamma\rangle  \right),\quad \gamma\in\mathcal{S}^p,
\end{equation*}%
where  recall  $K(\cdot )$
is a kernel,  $h$ a bandwidth and $K_{h}\left( \cdot \right) =K\left( \cdot /h\right) $.
Let $\widehat{v}_{n}^{2} (\cdot; \widehat a,\widehat b )$ be an estimate of the
variance of $nh^{1/2}Q_{n}(\cdot; \widehat a,\widehat b ) $ like in section \ref{est_var_sec}. Given  $B_p\subset \mathcal{S}^p$ with strictly positive Lebesgue measure in $\mathcal{S}^p$ that contains $\gamma_0^{(p)}$,  the least favorable direction $\gamma$ for $H_0$ is defined as
\begin{equation}  \label{bet_b}
\widehat{\gamma }_{n} = \arg \max_{\gamma\in B_p} \left[ n h
^{1/2} Q_{n}(\gamma; \widehat a,\widehat b )/\widehat{v}_{n}(\gamma; \widehat a,\widehat b ) - \alpha _{n} \mathbb{I}_{\left\{
\gamma \neq \gamma_0^{(p)}  \right\}} \right] \;.
\end{equation}
The test statistic  is then
\begin{equation}\label{test_stat_b}
T_{n} = n h^{1/2} \frac{Q_{n}(\widehat{\gamma}_{n}; \widehat a,\widehat b)}{
\widehat{ v}_{n} (\widehat{\gamma}_{n}; \widehat a,\widehat b)} \; .
\end{equation}
We will show that an asymptotic $\alpha$-level test is given by $\mathbb{I} \left(
T_{n}\geq z_{1-a} \right) $, where $z_{a}$ is the
$(1-a)-$quantile of the standard normal distribution.

To derive  the standard normal behavior of the test statistic under the null, we will show that under suitable conditions
\begin{equation}\label{eq_test_lin}
\sup_{\gamma\in \mathcal{S}^p} nh^{1/2}| Q_{n}(\gamma; \widehat a,\widehat b ) - Q_{n}(\gamma )|=o_{\mathbb{P}}(1) \;\;\; \text{and } \;\;\; \sup_{\gamma\in \mathcal{S}^p} | \widehat{ v}_{n} (\gamma; \widehat a,\widehat b)/\widehat{ v}_{n} (\gamma) - 1| =o_{\mathbb{P}}(1),
\end{equation}
with $Q_{n}(\gamma )$ and $\widehat{ v}_{n} (\gamma)$ defined like in section \ref{sec3}, that is we will bring the problem back to the case where the errors $U_i$ are observed.


\begin{lem}\label{diff_fund}
Assume the conditions of Theorem \ref{as_law} are met, $\mathbb{E}(|U|^m)<\infty$ for any $m\geq 1,$ and  $\int_0^1 \mathbb{E}[X^2(t)] dt<\infty.$ Let $\widehat b \in L^2[0,1]$ be an estimator of $b$ such that $\|\widehat b - b \|_{L^2} = O_{\mathbb{P}} (n^{-\rho})$ for some $3/8<\rho\leq 1/2$.  Moreover, suppose that the bandwidth $h$ is such that $n^{1-2\zeta}h^{1/2} \rightarrow 0$ for some $3/8 <\zeta < \rho.$
Then, under  the hypothesis $H_0$ the uniform rates of convergence in (\ref{eq_test_lin}) holds true.
\end{lem}

At this stage it is worthwhile to notice an important difference between the functional data framework and the finite-dimension case. In the later  case the parameters of a parametric regression model could be easily estimated at the usual rate $O_{\mathbb{P}}(n^{-1/2})$ which makes that the equivalences (\ref{eq_test_lin}) to hold without any further conditions on the model. In the functional covariate and functional parameter case, the rate of $\|\widehat b - b\|$ depends on the regularities of the covariate and of the slope parameter and is in general less than $O_{\mathbb{P}}(n^{-1/2})$, see Hall and Horowitz (2007), Crambes, Kneip and Sarda (2009).
To make the differences $\widehat U_i - U_i $ sufficiently small and hence to preserve the standard normal critical values for $T_{n}$ defined in (\ref{test_stat_b}) one has to pay a price on the bandwidth $h$: slower rates of $\|\widehat b - b\|_{L^2}$ will require faster decreases for $h$, and this will result in a loss of power against sequences of local alternatives. Below, we will investigate these aspects in more detail in the case  where the slope is estimated using the functional principal component approach.

\begin{theor}
\label{as_law2}  Under the conditions of Lemma \ref{diff_fund} and if the hypothesis $H_0$ holds true, the law of the test statistic $
T_{n} $ is asymptotically standard normal. Consequently the test given by $\mathbb{I}(T_n \geq z_{1-a})$, where $z_a$ is the $(1-a)-$quantile of the standard normal distribution,  has asymptotic level  $a.$
\end{theor}

The proof of this theorem is a direct consequence of Lemma \ref{diff_fund} and the arguments we used for Theorem \ref{as_law}, therefore we will omit it.

There are several possibilities to estimate the parameters of a functional linear model.
Let us investigate our test in the case where the estimate $\widehat b$ is obtained using the functional principal component analysis (PCA) approach, which is a standard approach for estimating the slope $b$; see for instance Ramsay and Silverman (2005) and Hall and Horowitz (2007). For the sake of completeness, let us briefly recall this estimation procedure.
Let  $\mathcal{K}(u,v)=\mbox{\rm Cov}(X(u),X(v))$, $\overline{X}_n=n^{-1}\sum_{i=1}^n  X_i$ and
$$ \widehat{\mathcal{K}}(u,v)=\sum_{i=1}^n (X_i(u)-\overline{X}_n(u))(X_i(v)-\overline{X}_n(v)).$$
Write the spectral expansions of $\mathcal{K}$ and $\widehat{\mathcal{K}}$ as
$$\mathcal{K}(u,v)=\sum_{j=1}^{\infty} \theta_j \phi_j(u) \phi_j(v), \qquad  \widehat{\mathcal{K}}(u,v)=\sum_{j=1}^{\infty} \widehat{\theta}_j \widehat{\phi}_j(u) \widehat{\phi}_j(v),$$
where
$$\theta_1 > \theta_2 > \cdots > 0, \qquad \widehat{\theta}_1 \geq \widehat{\theta}_2 \geq \cdots \geq 0$$
are the eigenvalues sequences of the operators with kernel $\mathcal{K}$ and $\widehat{\mathcal{K}}$, respectively, and $\phi_1, \phi_2, \ldots$ and $\widehat{\phi}_1, \widehat{\phi}_2, \ldots$ are the respective orthonormal eigenfunctions sequences. The linear operator corresponding to $\mathcal{K}$ is defined by $(\mathcal{K}f)(u)=\int \mathcal{K}(u,v) f(v) \textrm{d}v$. We have
$\mathcal{K}b=g$
where $g(u)=\esp [(Y-\esp(Y))(X(u)-\esp (X(u)) ]$. This suggests the estimator
\begin{equation}\label{est_pca_b}
\widehat{b}(t)=\sum_{j=1}^m \widehat{b}_j \widehat{\phi}_j(t), \quad t\in [0,1],
\end{equation}
where the truncation point $m$ is a smoothing parameter, $\widehat{b}_j=\widehat{\theta}^{-1}_j \widehat{g}_j$, $\widehat{g}_j=\langle \widehat{g}, \widehat{\phi}_j \rangle$ and
\begin{equation}\label{g_funct_i}
\widehat{g}(t)=n^{-1} \sum_{i=1}^n (Y_i-\overline{Y}_n)(X_i(t)-\overline{X}_n(t))
\end{equation}
with $\overline{Y}_n=n^{-1}\sum_{i=1}^n  Y_i$.

\setcounter{hp}{15}

For simplicity, hereafter the orthonormal basis $\mathcal{R} $ introduced in section
\ref{section1} is the basis composed of the sequence of orthonormal eigenfunctions
$\phi_1, \phi_2, \ldots$ of the covariance operator $\mathcal{K}$. Hence
$
X(t) = \sum_{j=1}^\infty x_j \phi_j(t),
$
where the random coefficients  $x_j = \langle  X, \phi_j \rangle $.
The following assumptions are standard conditions on the covariance operator $\mathcal{K}$ and  the slope parameter in the linear model, as could be found in Hall and Horowitz (2007).

\begin{hp}
\label{order}
\hspace{-0.1cm}
\begin{enumerate}
\item[(a)] The covariate $X$ has finite fourth moment, that is $\int_0^1 \mathbb{E}[X^4(t) ]dt <\infty$; moreover for some constant $C>1$, $\mathbb{E}[x_j - \mathbb{E}(x_j)]^4 \leq C \theta_j^2$ for all $j.$

\item[(b)] The errors $U_i$ are identically distributed, independent of $X_i$, with zero mean and finite variance.

\item[(c)] The eigenvalues $\theta_j$ of the covariance operator $\mathcal{K}$ satisfy $$ \theta_j-\theta_{j+1} \geq C^{-1}j^{-\alpha-1} \qquad \forall j \geq 1.$$

\item[(d)] The Fourier coefficients $b_j$ satisfy  $$|b_j| \leq Cj^{-\beta}   $$ and $\alpha >1$, $\frac{3}{2}\alpha + 2 < \beta$.
\end{enumerate}
\end{hp}

The condition $\frac{3}{2}\alpha + 2 < \beta$ replaces condition $\frac{1}{2}\alpha + 1 < \beta$ of Hall and Horowitz (2007) in order to conciliate the various requirements on the bandwidth $h$. For comparison,  see also Theorem 4.1 of Cai and Hall (2006) where is required $\beta\geq \alpha +2 .$
Hall and Horowitz (2007) show if Assumption \ref{order} hold true and if $m \asymp n^{1 / ( \alpha + 2\beta )}$, then
$$
\|\widehat b - b \|^2_{L^2} = O_{\mathbb{P}}\left( n^{- \, \frac{2\beta - 1}{\alpha+ 2 \beta}} \right),
$$
and this rate is optimal in a minimax sense. In this case $\rho =(2\beta - 1)/\{ 2(\alpha+ 2 \beta)\}$ and  the condition $\rho >3/8$ of Theorem \ref{as_law2}, which guarantees a non empty range for the bandwidth, becomes $\frac{3}{2}\alpha + 2 < \beta$, that is Assumption \ref{order}-(d).

\subsection{The behavior under the alternatives}

The alternatives of the functional linear model we consider are of the form
\begin{equation}
H_{1n}:\; Y_{in} = a+\langle b, X_i \rangle  + r_n \delta(X_i) + U^0_i ,\quad  n\geq 1, \quad\text{with} \;\; \mathbb{E}(U^0_i\mid X_i) = 0, \quad 1 \leq i \leq n,
\label{hh1bis}
\end{equation}
with $\delta(\cdot)$  an real-valued function such that $0< \mathbb{E}[\delta^4 (X)] <\infty$ and $r_n$, $n\geq 1$  a sequence of real numbers.

To be able to investigate the behavior of the test statistic under the alternatives, first we have
to analyze the behavior of $\widehat b$ the estimator of $b$.  To keep the paper at a reasonable length, hereafter we consider that  $\widehat b$ is estimator obtained through that functional PCA approach.
In Lemma \ref{leem_order_3} below we derive the rate of convergence of $\widehat b$ towards $b$ under the alternatives $H_{1n}$, provided that the function $\delta(\cdot)$ satisfies the orthogonality conditions
\begin{equation}\label{orth_cond_b}
\mathbb{E}[\delta(X)]=0\qquad \text{ and } \qquad \mathbb{E}[\delta(X)X]=0.
\end{equation}
Such orthogonality conditions are quite common in nonparametric testing, see for instance equation (3.11) in Guerre and Lavergne (2005), and they  allow to focus on the performance of the test to detect departures from the model.

\begin{lem}
\label{leem_order_3}
Assume that $X_{1}
,\ldots ,X_{n}$ are independent draws from  $X$, $\int_0^1 \mathbb{E}[X^2(t)] dt<\infty$ and condition (\ref{orth_cond_b}) hold true.
Let $\widehat b$ (resp. $\widehat{b}^0$) be the estimator defined in (\ref{est_pca_b}) obtained from data generated according to the model (\ref{hh1bis}) with a bounded sequence $r_n\geq 0$, $n\geq 1$ (resp. with $r_n=0$ for all $n\geq 1$).  Then
$$
\|\widehat{b}^0-\widehat{b}\|^2_{L^2} =  O_{\mathbb{P}}(r_n^2 n^{-1}) \sum_{j=1}^m \widehat{\theta}^{-2}_j.
$$
If in addition assumption \ref{order} hold true and $m \asymp n^{1 / ( \alpha + 2\beta )}$, then
$$\int_0^1 \{\widehat{b}(t)-b(t)\}^2 dt =O_{\mathbb{P}}\left(n^{-\frac{2\beta-1}{\alpha+2\beta}}\right)+ o_{\mathbb{P}}(r_n^2).$$
\end{lem}

\bigskip

Let us note that no moment condition for $U^0$ is needed for the proof of the first part of Lemma \ref{leem_order_3}. Moreover, let us point out that we will not need to investigate the convergence rate for the estimator of $a$ under the alternatives since by construction $$\widehat{a}-a=-\int_0^1\{\widehat{b}(t)-b(t)\}\overline{X}_n(t) dt+\overline{U}_n.$$
Now, we can analyze the behavior of the test statistics under the alternatives (\ref{hh1bis}). The estimated residuals $\widehat{U}_i$ can be decomposed
\begin{equation}\label{eq_alty}
\widehat{U}_i = U^0_i + r_n \delta(X_i) - \langle \widehat{b}-b,X_i-\overline{X}_n\rangle - r_n \overline{\delta(X)}_n - \overline{U^0}_n
\end{equation}

\begin{theor}
\label{altern_b}
Consider the sequence of alternative hypotheses (\ref{hh1bis}) with a nonzero function $\delta$ satisfying (\ref{orth_cond_b}) and $0<\mathbb{E}[\delta^4(X)]<\infty.$ Let $\widehat b \in L^2[0,1]$ be an estimator of the slope parameter $b.$ Suppose that the conditions of Theorem \ref{as_law2} are met with $U$ replaced by $U^0$. Moreover, assume that $\widehat b$, the sequence
$r_n $, $n\geq 1$, the sequence $\alpha_n $, $n\geq 1$ and the bandwidth $h$ satisfy the additional conditions:
\begin{enumerate}
\item[(i)] $r_n^2 n h^{1/2}  / \alpha_n \rightarrow \infty$;

\item[(ii)] $r_n^{-1} \| \widehat b - b  \|_{L^2} = o_{\mathbb{P}}(1)$;

\item[(iii)] $\alpha _{n}/\{ p^{3/2} \ln n \} \rightarrow
\infty.$
\end{enumerate}
Then the test based on $T_n$ defined in (\ref{test_stat_b}) will reject the functional linear regression model with probability tending to 1, provided there exists $ p\geq 1$ and  $\widetilde \gamma\in B_p$ such that at least one of conditions (1) to (3) of Theorem \ref{altern} holds true.
\end{theor}

If Assumption \ref{order} holds true,  condition (ii) of Theorem \ref{altern_b} indicates that the test could detect only local alternatives $H_{1n}$ that approach the null hypothesis slower than $n^{-(2\beta-1)/\{2(\alpha+2\beta)\}}. $ Meanwhile, in order to detect the fastest possible alternatives, the bandwidth should decrease to zero as slow as allowed by condition (i), that is faster than  $n^{-2(\alpha+1)/(\alpha+2\beta)} $ times a power of $\ln n$, provided the dimension $p$ and $\alpha_n$ increase as fast as a power of $\ln n$ such that  condition (iii) is met.

\section{Empirical analysis}\label{cesar_is_the_best}

\subsection{Bootstrap procedures}

To improve the critical values of the test statistic $T_n$ with small samples we consider a wild bootstrap procedure that can be applied in both cases we consider herein: the $U_i$'s are observed or the $U$'s are estimated by some $\widehat U_i$'s. A bootstrap sample is denoted by $U_i^b$ or $\widehat U_i^b$, $1\leq i\leq n$.
The wild bootstrap procedure we propose is inspired by Mammen (1993):  $U_i^b = Z_i U_i$ (resp. $\widehat U_i^b = Z_i \widehat U_i$), $1\leq i\leq n$, with $Z_i=V_i/\sqrt{2} +(V_i^2 - 1)/2$ and $V_i$ independent standard normal variables independent from the original observations.
A bootstrap test statistic is built
from a bootstrap sample as was the original test statistic. When this scheme is repeated many times, the
bootstrap critical value $z^\star_{1-\alpha, n}$ at level $a$ is the empirical $(1-a)-$th quantile of the bootstrapped test statistics.
This critical value is then compared to the initial test statistic.

\subsection{Simulation study}

An extensive simulation study was carried out. For reasons of brevity, only a summary of the main results and conclusions is given here. We will focus on the goodness-of-fit test of parametric models.

Let us start with the functional linear model, as considered in Section 4, given by
$$Y=a+\langle b,X\rangle +U$$
The function $X$ is drawn as a Brownian motion, $b\in L^2[0,1]$ and $a\in \mathbb{R}$ are parameters to be estimated, and $U=\delta(X)+U^0$, where $\delta(X)$ is the deviation from the null hypothesis, and $U^0$ is the error. For the parameters, $b(t)=1$ for all $t\in[0,1]$ and $a=0$ were taken as the true values.

A sample $(Y_1,X_1),\ldots(Y_n,X_n)$ of size $n=200$ will be drawn from this model, that is,
$$Y_i=a+\langle b,X_i\rangle + \delta(X_i)+U_i^0, \qquad\qquad1\leq i\leq n,$$
where $U_1^0,\ldots, U_n^0$ are independent standard normal variables, also independent of the $X_i.$

The first scenario will be the goodness-of-fit of the functional linear model versus a quadratic type deviation
$$\delta_Q(X)=c\left(\int_0^1\int_0^1 X(s) X(t)\,ds\,dt - 1/3\right)$$
where $c=0$ under the null hypothesis and $c=0.6$ under the alternative.
The second scenario will be the goodness-of-fit of the functional linear model versus a cubic deviation
\begin{equation} \label{alt_cubic}
\delta_c(X) = d\left(\int_0^1\int_0^1\int_0^1 X(s) X(t) X(z) \,ds\,dt\,dz - \int_0^1 X(t) dt\right)
\end{equation}
where $d=0$ under the null and $d=0.9$ under the alternative.
Note that the two functions $\delta_Q(\cdot)$ and $\delta_c(\cdot)$ satisfy the orthogonality conditions (\ref{orth_cond_b}).
In these two scenarios, the PCA estimator of the functional linear model, studied in Hall and Horowitz (2007), is used.

Let us recall that the Karhunen-Lo\`eve expansion of the Brownian motion $X$, is given by
$$X(t)=\sum_{j=1}^{\infty} x_j\, \frac{1}{(j-0.5)\pi}\,\sqrt{2}\sin\left((j-0.5)\pi t\right)$$
where $x_j$ are independent standard normal coefficients, ${\cal R}=\{\rho_j(t)=\sqrt{2}\sin((j-0.5)\pi t): j\in\{1,2,\ldots\}\}$ constitutes an orthonormal basis of eigenfunctions, and $1/((j-0.5)^2\pi^2)$ are eigenvalues.
We made use of this basis ${\cal R}$, and took different values of $p$, the number of basic elements. Other basis were also checked. The role played by the basis and the dimension $p$ consists of allowing to approximate both the covariate function $X$ and the alternative. A good basis is that which provides a good approximation with a small dimension $p$. The Karhunen-Lo\`eve basis is obviously a good basis to approximate the covariate function.

Several possible choices were studied for the privileged direction, $\gamma_0^{(p)}$. Here we present results with an uninformative one, with the same coefficients in all basic elements.

Different values for the penalization were considered. Since the statistic is standardized before penalization, natural values for $\alpha_n$ are 3, 4, 5 or 6. Small values of the penalization provide results that are similar to those obtained with the direction maximizing the standardized statistic, that is, $\arg \max_{\gamma\in{\cal S}^p} nh^{1/2}Q_n(\gamma)/\hat{v}_n(\gamma)$, while larger values of the penalization lead to results similar to those obtained with the chosen direction $\gamma_0^{(p)}$. The results presented here correspond to the penalization $\alpha_n=5$.

To compute the statistic for each direction, we used the Epanechnikov kernel, $K(x)=1-x^2$ for $x\in[-1,1]$. A grid of bandwidths was used in order to explore the effect of the bandwidth on the power of the test.

To estimate the conditional variance, the two estimators (\ref{var_est_tau}) and (\ref{var_est1b}) were considered. For the estimator (\ref{var_est1b}), a kernel estimator of the errors' conditional variance was used, with uniform kernel and bandwidth $h_v=0.5n^{-1/6}$. We observed a better power under the alternative with the estimator (\ref{var_est1b}), so the results will be given with this estimator.

For the optimization in the hypersphere ${\cal S}^p$, a grid of 300 points was used in the case of $p=3$ dimensions, and a grid of 1280 points in the case of $p=5$ dimensions. Aditionally, a local refinement of the optimum was used, with 9 points in dimension $p=3$ and 81 points in dimension $p=5$. For each original sample, we used 199 bootstrap samples to compute the critical value. One thousand original samples were used to approximate the percentages of rejection. The results shown in figures below were obtained with dimension $p=3$.

Figure \ref{Figure1} shows the empirical powers obtained for a grid of values of the bandwidth both under the null hypothesis of the functional linear and under the quadratic alternative. We observe that the power is not very much affected by the bandwidth around a possibly optimal value.

For purposes of comparison, the empirical power of the Horvath and Reeder (2011)'s test (HR test for brevity) is also shown. These authors proposed a test of significance of the quadratic effect under a functional quadratic model. Note that HR test is specially designed to detect quadratic alternatives to the linear model, as the one proposed here as the first scenario.
As expected, HR test is more powerful than our test, specially for dimension 1. This dimension represents the number of components in the estimation of the functional linear model, which in the case of the HR test coincides with the dimension used in the test statistic to estimate the quadratic deviation. HR test loses power when the dimension increases as a consequence of a bigger noise in the test statistic.

For our test, $m=3$ was used for the estimation of the functional linear model and $p=3$ was taken for the number of basic elements. Simulations were also carried out with other values of $m$ and $p$, and we observed that our test provides similar power when increasing each of these dimensions. The reason is that our test is not very much affected by the noise coming from increasing dimension, and this allows for a bigger dimension and consequently a better approximation of the deviation. It could be said that our test reaches a better trade-off between the noise coming from dimensionality and the approximation of the linear model and the alternative.

\begin{figure}[h!]
\begin{center}
\includegraphics[width=9cm]{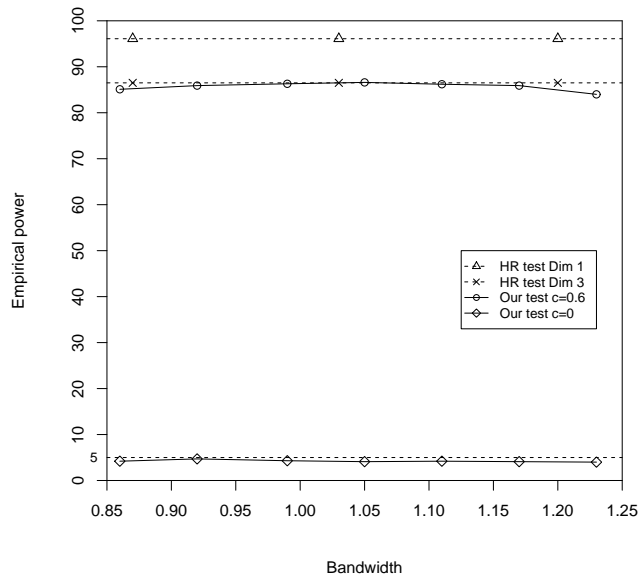}
\caption{Testing the functional linear model versus a quadratic alternative.}
\label{Figure1}
\end{center}
\end{figure}

Figure \ref{Figure2} shows the power of our test under the second scenario, where a functional linear model is tested versus a cubic alternative. As expected, HR test is not very powerful since it was not designed to detect this kind of alternative. The power of our test is good in a wide range of values of the bandwidth and the level is very well respected under the null hypothesis.

\begin{figure}[h!]
\begin{center}
\includegraphics[width=9cm]{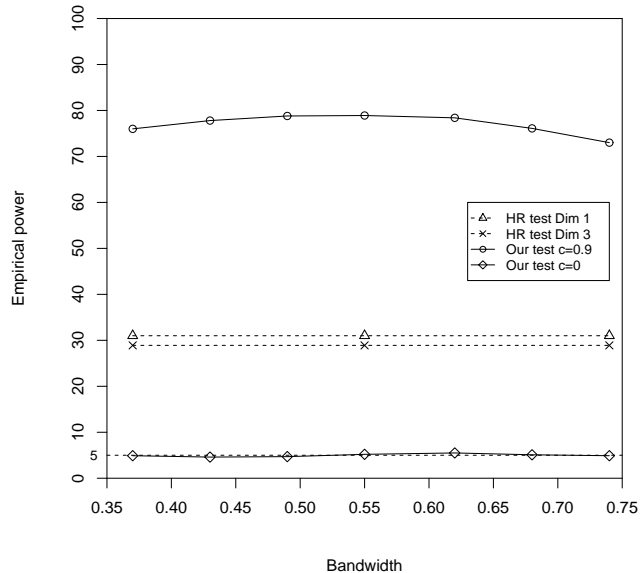}
\caption{Testing the functional linear model versus a cubic alternative.}
\label{Figure2}
\end{center}
\end{figure}

As an illustration of the behavior of our test for the goodness-of-fit of a more general parametric model, we considered the goodness-of-fit of the quadratic functional model, and obtained percentages of rejection under the null hypothesis and under the cubic alternative. That is, the simulated model would be
$$Y=a+\int_0^1 b(t) X(t) \,dt+\int_0^1\int_0^1 h(s,t) X(s)X(t)\, ds\,dt+\delta(X)+U^0$$
which consists of a quadratic functional model, as considered in Yao and M\"uller (2010) and Horvath and Reeder (2011), plus a deviation represented by the function $\delta(\cdot)$.
Here $b(t)=1$ for all $t\in[0,1]$, and $h(s,t)=0.6$ for all $s,t\in[0,1]$, which are the same linear and quadratic effects considered before. The deviation was chosen to be $\delta=\delta_c$, that is, the cubic deviation already considered in (\ref{alt_cubic}). Then, the idea will be to carry out a goodness-of-fit of the functional quadratic model, and evaluate its performance under the null $d=0$ and under the cubic alternative $d=0.9$. Results are shown in Figure \ref{Figure3}. To the best of our knowledge there is no parametric test for comparison in this situation.

The results show a certain conservative behavior for large bandwidths, while the power is generally good, with no much effect coming from the bandwidths. Results were obtained for different values of $p$ and the dimension of the parametric estimator, with generally good and expectable outcomes. Further investigation will be provided elsewhere.

\begin{figure}[h!]
\begin{center}
\includegraphics[width=9cm]{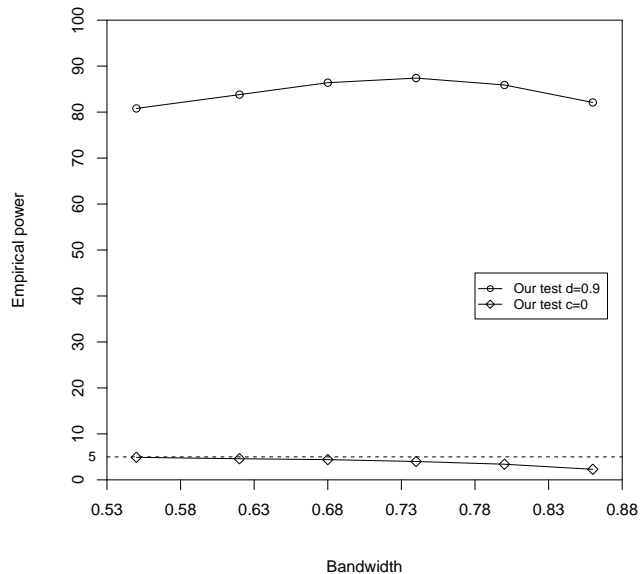}
\caption{Testing the functional quadratic model versus a cubic alternative.}
\label{Figure3}
\end{center}
\end{figure}

\subsection{Application to real data}

The test proposed here is applied to the data set collected by Tecator and available at http://lib.stat.cmu.edu/datasets/tecator.
The task is to predict the fat content of a
meat sample on the basis of its near infrared absorbance spectrum.
For each sample of finely chopped pure meat, a 100 channel spectrum of absorbances was recorded using a Tecator Infratec Food and Feed Analyzer, a spectrometer that works in the wavelength range 850-1050 nm. These absorbances can be thought of as a discrete approximation to the continuous record, $X_i(t)$. Also, for each sample of meat, the fat content, $Y_i$, was measured by analytic chemistry. The data set contains 240 samples of meat.

Yao and M\"uller (2010) proposed using a functional quadratic model to predict the fat content, $Y_i$, of a meat based on its absorbance spectrum, $X_i(t)$. Horvath and Reeder (2011) applied their parametric test to check whether the quadratic term is needed, versus the null hypothesis of a functional linear model. Their reached the conclusion that the quadratic effect is significant, and then, the functional quadratic model is needed.

We will apply the test proposed here, first to check the goodness-of-fit of the functional linear model, and later the goodness-of-fit of the functional quadratic model. Table 1 below contains the p-values corresponding to our test for different values of the bandwidth, the parameter $m$ for model estimation and the dimension $p$. We can conclude that both the functional linear and the functional quadratic models should be rejected for the Tecator data set.  This conclusion confirms the empirical results of Chen, Hall and M\"{u}ller (2011) who proposed an additive double index model. Indeed, the link functions estimated by Chen, Hall and M\"{u}ller do not show respective linear and quadratic patterns which indicates that the usual functional linear and the functional quadratic models do not provide a satisfactory fit.

$$\begin{tabular}{ccccccccccc}
&&\multicolumn{4}{c}{Linear model}&&\multicolumn{4}{c}{Quadratic model} \\
& $h$ & 0.18 & 0.30 & 0.44 & 0.59 & &  0.18 & 0.30 & 0.44 & 0.59 \\
\hline
$p=2$ & $m=1$ & 0.5 & 0.4 & 0.2 & 0.6 & & 2.4 & 1.4 & 1.6 & 3.3 \\
  & $m=2$ & 0.2 & 0.0 & 0.0 & 0.3 & & 0.6 & 0.3 & 0.0 & 0.7 \\
  & $m=3$ & 0.0 & 0.0 & 0.0 & 0.0 & & 0.0 & 0.0 & 0.0 & 0.0 \\
&&&&&&&&&& \\
$p=3$ & $m=1$ & 0.0 & 0.0 & 0.2 & 0.2 & & 0.0 & 0.1 & 0.1 & 0.0 \\
  & $m=2$ & 0.0 & 0.0 & 0.0 & 0.1 & & 0.2 & 0.0 & 0.1 & 0.0 \\
  & $m=3$ & 0.0 & 0.0 & 0.0 & 0.0 & & 0.0 & 0.0 & 0.0 & 0.0 \\ \hline
\end{tabular}$$
Table 1. p-values (in percentages) obtained by applying the new test to the Tecator data set.

\newpage

\section{Appendix}

\label{secproofs} \setcounter{equation}{0}

\subsection{Dimension reduction: proof of the fundamental lemma}

\begin{proof}[Proof of Lemma \ref{lem1}]
(A). The implication $ (1)\,\Rightarrow \,(2)$ is obvious. To prove $ (2)\,\Rightarrow \,(1)$, note first that for any
$\beta \neq 0$, the $\sigma -$field generated by $\langle X,\beta\rangle  $ is the
same as the $\sigma -$field generated by $\langle X,\beta \rangle /\|\beta
\|_{L^2} $. Next, by elementary properties of the
conditional expectation, we obtain that for any $\beta\in L^2[0,1] $, including
$\beta =0 $,
\begin{eqnarray}
0 &=& \mathbb{E}\left[ \exp \{i\langle X,\beta \rangle  \}\mathbb{E}(Z\mid \langle X,\beta\rangle )\right] \notag\\
&=& \mathbb{E}\left[ \exp \{i\langle X,\beta\rangle \}Z\right]\notag\\
& =&
\mathbb{E}\left[ \exp \{i\langle X,\beta\rangle \}\mathbb{E}(Z\mid X)\right]
\; .\label{mopz}
\end{eqnarray}
Write $Z=Z^{+}-Z^{-}$ where $Z^{+}$ and $Z^{-}$ are the
positive and negative parts of $Z$, and deduce that for any $\beta $, $%
\mathbb{E}\left[ \exp \{i\langle X,\beta\rangle \}\mathbb{E}(Z^{+}\mid X)\right] =
\mathbb{E}\left[ \exp \{i\langle X,\beta\rangle \}\mathbb{E}(Z^{-}\mid X)\right] $.
As distinct positive finite measures cannot have the same characteristic
function,  see for instance Theorem 3.1 of Parthasarathy (1967),
this implies that $\mathbb{E}(Z^{+}\mid X)=\mathbb{E}(Z^{-}\mid X)$
and hence $\mathbb{E}(Z\mid X)=0$ almost surely. For $ (2)\,\Rightarrow \,(3)$ it suffices to identify
$\gamma$ with an element in $L^2[0,1]$ of norm 1. To prove $ (3)\,\Rightarrow \,(1)$, fix arbitrarily
$\beta\in L^2[0,1]$, $\beta\neq 0$. For any $p\geq 1$, let  $\beta^{(p)}$
be the projection of $\beta$ on the subspace generated by the first $p$ elements of the basis  $\mathcal{R}$. For any $p$ sufficiently large such that $\|\beta^{(p)}\|=\|\beta^{(p)}\|_{L^2}>0$ we have
$\mathbb{E}(Z\mid \langle X,\beta^{(p)}\rangle ) = \mathbb{E}(Z\mid \langle X,\beta^{(p)}/\|\beta^{(p)}\|\rangle ) = 0,$ where for the last equality we use the fact that $\beta^{(p)}/\|\beta^{(p)}\|\in \mathcal{S}^p$ and (3). By elementary properties of the conditional expectation,
\begin{eqnarray*}
0 &=& \mathbb{E}\left[ \exp \{i\langle X,\beta^{(p)} \rangle  \}\mathbb{E}(Z\mid \langle X,\beta^{(p)}\rangle )\right]\\
&=& \mathbb{E}\left[ \exp \{i\langle X,\beta \rangle \}Z \exp \{i\langle X,  \beta^{(p)} - \beta \rangle \}\right] \\
&=& \mathbb{E}\left[ \exp \{i\langle X,\beta \rangle \}\mathbb{E}(Z\mid X) [\exp \{i\langle X,  \beta^{(p)} - \beta \rangle \} - 1]\right]\\
&&+ \mathbb{E}\left[ \exp \{i\langle X,\beta \rangle \}\mathbb{E}(Z\mid X)\right].
\end{eqnarray*}
From the Taylor expansion with integral reminder and elementary calculus, one obtains that  $\forall x\in\mathbb{R}$,  $|\exp (ix)-1|\leq  \min\{|x|,2\}$.
From this and the Cauchy-Schwarz inequality, deduce that for any $p$,  $$|\exp \{i \langle X,  \beta^{(p)} - \beta \rangle \} - 1| \leq
\min\{\|X\|_{L^2}\|\beta^{(p)} - \beta \|_{L^2},2\}.$$ Since $\| \beta^{(p)} - \beta \|_{L^2}\rightarrow 0$ when $p\rightarrow\infty$, by Lebesgue Dominated Convergence Theorem it follows that  necessarily $$\mathbb{E}\left[ \exp \{i\langle X,\beta \rangle \}\mathbb{E}(Z\mid X) \right]=0.$$ Since $\beta\in L^2[0,1]$ was arbitrarily fixed, apply again the arguments we used after equation (\ref{mopz}) to deduce that (1) hold true. The equivalence  $(3)\,\Leftrightarrow \,(4)$ follows from Lemma 2.1-(A) of Lavergne and Patilea (2008) applied for each $p$.

(B). From  (A)-4 above, there exists some $p_0\geq 1$ such that $\mathbb{P} [\mathbb{E} (Z\mid X^{(p_0)}) = 0] <1$. On the other hand,  by the property of iterated expectations, for any $p > p_0$,
$$
\mathbb{E} (Z\mid X^{(p_0)}) = \mathbb{E}[ \mathbb{E} (Z\mid X^{(p)}) \mid X^{(p_0)}] \;.
$$
Thus necessarily $\mathbb{P} [\mathbb{E} (Z\mid X^{(p)}) = 0] <1,$  $\forall p>p_0$.
Fix arbitrarily $p > p_0$ and notice that for any $b\in[-1,1]$,
$$\{\gamma\in\mathcal{S}^p : \mathbb{E}(Z \mid \langle X, \gamma \rangle )=0 \,\, a.s.\, \}\subset    \{\gamma\in\mathcal{S}^p : \mathbb{E}(Z \exp\{ b \langle X , \gamma \rangle\} )=0 \}.$$ The expectations in the sets in the last display are well-defined since $$\mathbb{E}(|Z \exp\{ b \langle X , \gamma \rangle\}|)\leq
\mathbb{E}(|Z |\exp\{ |b| |\langle X , \gamma \rangle|\})\leq \mathbb{E}(|Z |\exp\{  \| X \|\}) <\infty.$$
Let us notice that
$$
\{ b\widetilde \gamma : b\in[-1,1], \; \gamma\in\mathcal{S}^p ,\; \mathbb{E}(Z \exp\{ b \langle X , \gamma \rangle\} )=0  \}\subset \widetilde A_p
$$
where
$$\widetilde A_p : = \{ \widetilde \gamma\in\mathbb{R}^{p}: \|\widetilde \gamma \|\leq 1, \mathbb{E}(Z \exp\{  \langle X^{(p)} , \widetilde \gamma \rangle\} )=0 \}.
$$
Thus, to prove  (B)
it suffice to show that the set $\widetilde A_p$
has Lebesgue measure zero in $\mathbb{R}^{p-1}$ and is not dense in the unit ball of $\mathbb{R}^{p-1}$.\footnote{An easy way to check that it is indeed sufficient to derive such properties for
$\widetilde A_p$ is to represent the sets in the hyperspherical coordinates.}

For these purpose, we will use the following property: if $W_1$ and $W_2$ are real-valued random variables such that
$\mathbb{E} (|W_1| \exp \{ a|W_2| \}) <\infty $ for some $a>1$, then
\begin{equation}\label{qsd}
\mathbb{P} (\mathbb{E}(W_1\mid W_2)=0 ) <1 \; \Longrightarrow \; \text{ the set } \{  |b| < a : \mathbb{E} (W_1 \exp\{bW_2 \}) = 0\}\; \text{is empty or finite.}
\end{equation}
To prove this property, decompose  $W_1 = W_1^+ - W_1^-$ and use the positive part $W_1^+$ to define
$$
\lambda^+ (b)= \mathbb{E} (W_1^+ \exp\{ b W_2 \}) = \mathbb{E} (\mathbb{E}(W_1^+ \mid W_2) \exp\{ b W_2 \}) = \int_{\mathbb{R}}\exp\{b w\} d\mu^+ (w),
$$
$|b|<a$, where $d\mu^+ (w) = \mathbb{E}(W_1^+ \mid W_2=w) dF_{W_2}(w)$ and $F_{W_2}$ is the probability distribution function of $W_2$. Use the negative part of $W_1$ to define $\lambda^- (b)$, $|b|<a$ similarly. Since $W_1$ is integrable,  $\mu^- (\mathbb{R}) , \mu^+ (\mathbb{R}) <\infty.$ The functions  $ \lambda^- (\cdot)/ \mu^- (\mathbb{R})$ and $ \lambda^+ (\cdot) / \mu^+ (\mathbb{R})$ are the Laplace transforms of the probability distributions $\mu^-/\mu^- (\mathbb{R})$ and $\mu^+/\mu^+ (\mathbb{R})$. The condition $\mathbb{E} (|W_1| \exp \{a |W_2| \}) <\infty $ implies that these Laplace transforms, and hence $ \lambda^+ (\cdot)$ and $ \lambda^- (\cdot)$, are (real) analytic on the domain $(-a,a)$. See for instance Proposition 8.4.4 in Chow and Teicher (1997).  Notice that the set in (\ref{qsd}) is the set of $b\in(-a,a)$ for which $ \lambda^- (b) =  \lambda^+ (b).$
If $\mathbb{P} (\mathbb{E}(W_1\mid W_2)=0 ) <1$, $ \lambda^+ (\cdot)$ and $ \lambda^- (\cdot)$ cannot coincide on $(-a,a)$. Thus the set in (\ref{qsd}) contains only isolated points $b$ from the interval $(-a,a)$, which means that it is necessarily empty or finite.

Now, recall that  we want to investigate the cardinality of the  $\widetilde A_p,$ subset of the unit ball of $\mathbb{R}^{p}.$ From (A), there exists $\widetilde \gamma \in \mathcal{S}^{p}$ such that $\mathbb{P} (\mathbb{E}(Z \mid \langle X^{(p)},\widetilde  \gamma \rangle ) = 0 )<1.$ Then,  property (\ref{qsd}) applied with some $a>1$, $W_1=Z$ and $W_2 = \langle X^{(p)},\widetilde  \gamma \rangle$ implies that the set
$$
\{  |b| < a : \mathbb{E} (Z \exp\{\langle X^{(p)}, b \widetilde \gamma  \rangle \}) = 0\}
$$
is empty or finite. Deduce that there exists $\upsilon^\star$ in the unit ball of $\mathbb{R}^{p}$, arbitrarily close to the origin, in particular with $\|\upsilon^\star\|<1/2$,  such that  $\mathbb{E} (Z \exp\{\langle X^{(p)}, \upsilon^\star  \rangle \}) \neq 0.$ Next, we adapt the lines of the proof of Lemma 1 in Bierens (1990). Let $Z^\star = Z \exp\{\langle X^{(p)}, \upsilon^\star  \rangle \}$. By construction,
$\mathbb{P} (\mathbb{E}(Z^\star \mid x_1,\cdots, x_l) = 0 ) <1,$ for $l=1,\cdots,p.$ Define the sets
$$
A^\star _l = \{ (t_1,\cdots,t_l)\in\mathbb{R}^l : \|(t_1,\cdots,t_l)\| \leq 3/2, \mathbb{E}(Z^\star  \exp\{ (x_1 t_1+\cdots+x_l t_l)\})=0 \},
$$
$l=1,\cdots,p$. Since $|t_1x_1+ \langle X^{(p)},\upsilon^\star\rangle|\leq \{|t_1| +  \| \upsilon^\star \|\} \| X^{(p)} \|< 2\| X^{(p)} \|,$ deduce from property  (\ref{qsd}) applied with $a=2$, $W_1 = Z$ and $W_2 = \exp\{t_1x_1+ \langle X^{(p)},\upsilon^\star\rangle \}$ that the set $A^\star _1$ is empty or finite.
Now, define the set
$$
A^{\star\star}_2 (t_1) = \{ |t_2|\leq 3/2 :  \mathbb{E}(Z^\star  \exp\{ x_1 t_1 \} \exp\{ x_2 t_2 \} )=0 \}.
$$
If $t_1\notin A^\star _1$, replace $Z^\star$ by $Z^\star  \exp\{ x_1 t_1\}$ and use again property (\ref{qsd}) with $a=7/2$, $W_1 = Z$ and $W_2 = \exp\{t_1x_1+ t_2 x_2 +\langle X^{(p)},\upsilon^\star\rangle \}$ to deduce that the set $A^{\star\star}_2 (t_1)$ is empty or finite. This means that $A^\star _2$ is contained in the union of some sets $B^\prime \times \mathbb{R}$ and
$\mathbb{R} \times B^{\prime\prime} $ where $B^{\prime}$ and $B^{\prime\prime}$ are empty or finite. Repeat the arguments with $l=3,\cdots,p$ and deduce that $A^\star _p$ has Lebesgue measure zero in $\mathbb{R}^{p}.$ Since the norm of $\upsilon^\star$ could be taken arbitrarily small such that
$\widetilde A_p \subset A^\star _p, $ we can now easily deduce that $\widetilde A_p$ has Lebesgue measure zero in the unit ball of $\mathbb{R}^{p}$. The fact that $\widetilde A_p$ is not dense  in the unit ball of $\mathbb{R}^{p}$ is a direct consequence of the fact that $A^\star _p$ intersected with unit ball of $\mathbb{R}^{p}$ is not dense.
\end{proof}


\subsection{Rates of convergence: technical lemmas}

For $\nu$ a probability measure on a sample space, $\mathcal{F}$ a class of functions and $\varepsilon >0$, let  $N(\varepsilon, \mathcal{F}, L^2(\nu))$, denote the covering number, that is the minimal number of balls of radius $\varepsilon$ in $L^2(\nu)$ needed to cover $\mathcal{F}$. See Van der Vaart and Wellner (1996) or Kosorok (2008) for the definitions. For real
random variables, $A_n\asymp_{\mathbb{P}} B_n$ means that there exists a constant $C>1$ such that  $\mathbb{P}(1/C \leq A_n/B_n \leq  C)$ goes to 1 when $n$ grows. In the following $C, C_1, c, c_1,\cdots$ represent constants that may change from line to line.

\begin{lem}\label{entropy}
For any $p\geq 1$, let $$\mathcal{F}_{1p} = \{(v_1,v_2)\mapsto K(h^{-1}\langle v_1 -v_2,\gamma\rangle) : v_1,v_2\in\mathbb{R}^p, \gamma\in\mathcal{S}^p, h>0 \}$$
and
$$\mathcal{F}_{2p} = \{v\mapsto \mathbb{E}[K(h^{-1}\langle X -v,\gamma\rangle)\mid  : v\in\mathbb{R}^p, \gamma\in\mathcal{S}^p, h>0 \}.$$ If Assumption \ref{K}-(a) holds, there exist constants $c_1,c_2,c_3>0$ such that for any $p\geq 1$ and $0<\varepsilon<1$ and any $\nu_1$ probability measure on $\mathbb{R}^{p}\times \mathbb{R}^{p}$ and
$\nu_2$ probability measure on $\mathbb{R}^{p}$,
\begin{equation}\label{cov_entr}
N(\varepsilon, \mathcal{F}_{jp\,} , L^2(\nu_j)) \leq c_1 (c_2/\varepsilon)^{c_3 p}, \quad j=1,2.
\end{equation}
\end{lem}
\begin{proof}
Since $K$ can be written as a difference of two monotone functions, the result for $\mathcal{F}_{1p}$ is an easy consequence of the Theorem 9.3, Lemmas 9.6 and 9.9 of Kosorok (2008) and Lemma 16 of Nolan and Pollard (1987); see also their Lemma 22-(ii). For $\mathcal{F}_{2p}$, use the bound for $\mathcal{F}_{1p}$ and Lemma 20 of Nolan and Pollard (1987).
\end{proof}

\bigskip

\begin{lem}\label{tech_u_stat}
Let Assumptions \ref{D}   and \ref{K} hold true and let $l$ be some strictly positive integer.
For each $n$ and $p$ that may depend on $n$, define the $U-$processes
$$
V_n^{(k_1,k_2)} (\gamma;l) = \frac{1}{n(n-1) h}\sum_{1\leq i\neq j\leq n} U_i^{k_1} U_j^{k_2} K_{h}^l \left(\langle X_{i}-X_{j},\gamma\rangle \right),\quad \gamma\in B_p, \;\;k_1,k_2\in\{ 0,2\}.
$$
Then $$\sup_{\gamma\in B_p} |V_n^{(0,0)} (\gamma;l)|  \asymp_{\mathbb{P}} 1,  \quad \sup_{\gamma\in B_p} \{1/|V_n^{(2,2)} (\gamma;l)| \} = O_{\mathbb{P}} (1)  \; \text{ and  } \;\sup_{\gamma\in B_p} |V_n^{(2,0)} (\gamma;l)| = O_{\mathbb{P}} (1).$$
\end{lem}

\begin{proof}
To simplify the writings, we write  $V_n^{(0)}$ (resp. $V_n^{(2)}$) instead of $V_n^{(0,0)}$
(resp. $V_n^{(2,2)}$). First consider the case $k_1 = k_2 =0.$ Hoeffding's decomposition allows us to decompose the centered  $U-$processes $h V_n^{(0)}(\gamma;l) - \mathbb{E}[h V_n^{(0)}(\gamma;l)]$ as a sum of two degenerate $U-$processes $V_{1n}^{(0)}(\gamma;l)$ and $V_{2n}^{(0)}(\gamma;l)$, $\gamma\in B_p$, of respective orders 1 and 2 that are indexed by families of functions obtained by finite sums of sets like $\mathcal{F}_{1p}$ and $\mathcal{F}_{2p}$ in Lemma \ref{entropy} above.  By Lemma 16 of Nolan and Pollard (1987), deduce that the families indexing  $V_{1n}^{(0)}(\gamma;l)$ and $V_{2n}^{(0)}(\gamma;l)$ are families with covering numbers bounded by polynomials in $1/\varepsilon$ with coefficient and order depending on $c_1$, $c_2$ and $c_3$ but independent of $n$ and $p$. (When $l >1$, $K$ should be replaced by $K^l$ in the definitions of $\mathcal{F}_{1p}$ and $\mathcal{F}_{2p}$, but given the properties of $K(\cdot)$ this has no impact on the conclusion.) Next, by Theorem 2 of Major (2006), $\sup_{\gamma\in B_p} |V_{2n}^{(0)} (\gamma;l) | = O_{\mathbb{P}} (n^{-1} h^{1/2} p^{3/2} \ln n)$; see the proof of our Lemma \ref{leem1} for an example of application of the result of Major (2006). On the other hand,  by Theorem 2.14.1 or Theorem 2.14.9 of van der Vaart and Wellner (1996), we have $\sup_{\gamma\in B_p} |V_{1n} ^{(0)} (\gamma;l) | = O_{\mathbb{P}} (n^{-1/2} p^{1/2}).$ Gathering  the rates and using Assumption \ref{K}-(b,c) we deduce that $ V_n^{(0)}(\gamma;l) - \mathbb{E}[V_n^{(0)}(\gamma;l)]=o_{\mathbb{P}}(1)$,  uniformly in $\gamma\in B_p$. Now, it remains to show that there exist constants $c_1,c_2>0$ such that $c_1\leq \mathbb{E}[V_n^{(0)}(\gamma;l)] = \mathbb{E}[h^{-1} K_{h}^l \left(\langle X_{1}-X_{2},\gamma\rangle \right)]\leq c_2 $, $\forall \gamma \in B_p$ and  $h$ sufficiently small.
Using the properties of the Fourier and inverse Fourier transforms, Fubini theorem, the independence of $X_1$ and $X_2$ and Plancherel theorem
\begin{eqnarray}\label{ffr}
\mathbb{E}[h^{-1} K_{h}^l \left(\langle X_{1}-X_{2},\gamma\rangle \right)] &=& (2\pi)^{-1/2}\mathbb{E} \int_{\mathbb{R}} \exp\{i t \langle X_{1},\gamma\rangle\}\exp\{ -i t \langle X_{2},\gamma\rangle\}\mathcal{F}[K^l] (t) dt\notag\\
&=& (2\pi)^{1/2} \int_{\mathbb{R}} |\mathcal{F}[f_\gamma](t)|^2\mathcal{F}[K^l] (ht) dt\notag \\
&\leq & (2\pi)^{1/2} \int_{\mathbb{R}} |\mathcal{F}[f_\gamma](t)|^2 dt
=
(2\pi)^{1/2} \int_{\mathbb{R}} f^2_\gamma(x) dx.
\end{eqnarray}
Assumption \ref{ass_dc}-(c)(i) guarantees that $\mathbb{E}[h^{-1} K_{h}^l \left(\langle X_{1}-X_{2},\gamma\rangle \right)]$ is uniformly bounded from above. On the other hand, using the positiveness of $\mathcal{F}[K]$ (hence of $\mathcal{F}[K^l]$),  the fact that $\mathcal{F}[K^l]$ is necessarily bounded away from zero on compact intervals,  the previous display and   Assumption \ref{ass_dc}-(c)(ii), deduce that there exists constants $c_3$ and $c_4$ such that $\forall p\geq 1$, $\forall \gamma\in B_p$ and $\forall h\leq 1$ (say),
$$
\mathbb{E}[h^{-1} K_{h}^l \left(\langle X_{1}-X_{2},\gamma\rangle \right)]\geq c_3  \int_{|t|\leq \epsilon} |\mathcal{F}[f_\gamma](t)|^2 dt \geq c_4>0.
$$
In the case $k_1=k_2=2$, by Assumption \ref{D}-(b),  $\mathbb{E}(V_n^{(2)}(\gamma;l))\geq \underline{\sigma}^{\,4} \mathbb{E}[h^{-1} K_{h}^l \left(\langle X_{1}-X_{2},\gamma\rangle  \right)] ,$ $\forall \gamma.$ Next use again Hoeffding's decomposition  for
$V_n^{(2)}(\gamma;l) - \mathbb{E}(V_n^{(2)}(\gamma;l)).$ The degenerate $U-$statistics of order 1 and 2 can be treated with the same arguments as above.
Deduce that $1/V_n^{(2)}(\gamma;l)$ is uniformly bounded in probability.  The case $k_1=0$ and $k_2=2$ could be handled with similar arguments.
\end{proof}

\bigskip

\begin{lem}\label{tech_emp_proc}
Under the conditions of Lemma \ref{diff_fund}, for any $\epsilon >0,$
\begin{equation}\label{vawe2}
\sup_{\gamma\in\mathcal{S}^p, \;\; t\in\mathbb{R}} \left|\frac{1}{n} \sum_{i=1}^n U_i  K_h (\langle X_i ,\gamma\rangle - t)\right|=
O_{\mathbb{P}}(p^{1/2}n^{-1/2+\epsilon}h^{1/2}).
\end{equation}
Moreover, there exists $a >0$ such that for any $\epsilon >0,$
\begin{equation}\label{vawe1}
\sup_{\gamma\in B_p} \left|\frac{1}{n(n-1)h} \sum_{1\leq i\neq j\leq n} U_i  K_h (\langle X_i - X_j ,\gamma\rangle)\right| = O_{\mathbb{P}}(p^{1/2}n^{-1/2+\epsilon}h^{-1/2 +a }).
\end{equation}
\end{lem}

\begin{proof}
Let $\mathcal{G} $ be a family of  functions $|g|\leq 1$ with covering number bounded by $(K/\varepsilon)^{V}.$ With the notation of van der Vaart and Wellner (1996), let $\mathbb{G}_n g$, $g\in\mathcal{G}$ be the empirical process indexed by $\mathcal{G}$ and let $\| \mathbb{G}_n \|_{ \mathcal{G} } = \sup_{\mathcal{G}}|\mathbb{G}_n g|.$  From Theorem 2.14.16 of van der Vaart and Wellner, after tracing the constants in the proof,
there exists $C>0$ independent of $n$ and $p$ such that for any $\delta>0$, $\exists C_\delta >0$ independent of  $n$ and $p$ such that
\begin{equation}\label{vwww}
\mathbb{P}\left( \| \mathbb{G}_n \|_{ \mathcal{G} } >t\right)\leq C \left(\frac{C_\delta}{\sigma} \right)^{2V} \left( 1\vee \frac{t}{\sigma} \right)^{3V+\delta} \exp\left\{ -\frac{1}{2} \frac{t^2}{\sigma^2 + (3+t)/\sqrt{n}} \right\}
\end{equation}
for every $t>0$ and every $\sup_{\mathcal{G}} Var (g)\leq  \sigma^2 \leq 1.$ Fix an arbitrary $\epsilon >0$, use  the covering number of $\mathcal{F}_{1p}$ in Lemma \ref{entropy} and apply this inequality with $g(U,X) = n^{-\epsilon/2} U \mathbb{I} (|U| \leq n^{\epsilon/2})  K_h (\langle X ,\gamma\rangle - t),$  $\sigma = c h^{1/2}$ and $t=\widetilde t p^{1/2} n^{-\epsilon/2} \ln^{1/2} n.$  Next derive the rate of  the reminder $n^{-1}\sum_{i=1}^n U_i (|U| > n^{\epsilon/2}) K_h (\langle X_i ,\gamma\rangle - t)$ taking absolute values, recalling that $\mathbb{E}(|U|^m)<\infty$ for every $m\geq 1,$ and using Markov inequality. For the second quantity, apply the Hoeffding decomposition to the second order $U-$statistics defined by
$h(U_i,X_i, U_j,X_j) =  n^{-\epsilon/2} U_i \mathbb{I} (|U_i| \leq n^{\epsilon/2})  K_h (\langle X_i -X_j ,\gamma\rangle).$
For the degenerate $U-$statistics of order 2 multiply by $h$ and proceed like in Lemma \ref{leem1}. To apply inequality (\ref{vwww}) for the empirical process in  Hoeffding decomposition we need a bound for the variance of the conditional expectation $\mathbb{E}\left[h^{-1}K_h (\langle X_i - X_j ,\gamma\rangle)\mid X_i \right]$.  Let $\delta>0$ such that $\int_{\mathbb{R}} f_\gamma ^{2+\delta} \leq  C$, $\forall \gamma\in B_p,$ for some constant $C$. By a change of variables, the boundedness of $K$, Jensen inequality, and again a change of variables we have for some constant $C^\prime,$
\begin{eqnarray*}
\mathbb{E}\{\mathbb{E}^2\left[h^{-1}K_h (\langle X_i - X_j ,\gamma\rangle)\mid X_i \right] \} &=&  \int_{\mathbb{R}}\left[\int_{\mathbb{R}} f_\gamma (u-th)K(t) dt \right] ^2f_\gamma (u) du\\
&\leq&  \int_{\mathbb{R}}\left(\int_{\mathbb{R}} f_\gamma^{2+\delta} (u-th)dt\right)^{2/(2+\delta)}
 f_\gamma (u) du\\
& \leq & C^\prime h^{- 2/(2+\delta)}.
\end{eqnarray*}
Use the covering number of $\mathcal{F}_{2p}$ in Lemma \ref{entropy} and inequality (\ref{vwww}) to deduce
\begin{equation*}
\sup_{\gamma\in B_p} \!\left|\frac{1}{n} \!\!\sum_{1\leq i\leq n}\!\! U_i \mathbb{I} (|U_i| \!\leq\! n^{\epsilon/2}) \mathbb{E}\left[h^{-1} \!K_h (\langle X_i \!-\! X_j ,\gamma\rangle)\! \mid \!X_i \right]\right| \!=\! O_{\mathbb{P}}(p^{1/2}h^{- 1/(2+\delta)}n^{-1/2+\epsilon/2}\ln^{1/2} \!n ).
\end{equation*}
Use  Markov inequality to bound $n^{-1}\sum_{i=1}^n U_i (|U| > n^{\epsilon/2})\mathbb{E}[h^{-1} K_h (\langle X_i -X_j,\gamma\rangle )\mid X_i]$ and hence complete the proof.
\end{proof}

\subsection{Testing for no-effect: proofs of the asymptotic results}

Let
\begin{equation}\label{drev}
v_{n}^2 (\gamma_0^{(p)}) = \frac{2}{n(n-1)h}\!\sum\limits_{j\neq i}
\sigma^{2}_{\gamma_0^{(p)}}( \langle X_{j}, \gamma_0^{(p)} \rangle) \sigma^{2}_{\gamma_0^{(p)}}( \langle X_{j}, \gamma_0^{(p)} \rangle)  K_{h}^{2}\left(
\langle X_{i}-X_{j},  \gamma_0^{(p)} \rangle  \right) .
\end{equation}

\begin{lem}\label{rav_u}
Let Assumptions \ref{D}, \ref{K} and hypothesis $H_0$ hold true.
Then $\widehat \tau_n^2 (\gamma_0^{(p)}) = \tau_n^2 (\gamma_0^{(p)}) \{ 1+o_{\mathbb{P}} (1)\} = v_n^2 (\gamma_0^{(p)}) \{ 1+o_{\mathbb{P}} (1)\}$. Moreover,
$\widehat{v}_{n}^{2}  = \tau_{n}^{2} (\gamma_{0}^{(p)})\{1 + o_{\mathbb{P}} (1)\}$, with  $\widehat{v}_{n}^{2}$ defined in (\ref{var_est1b}), provided that condition (\ref{nonp}) holds true.
\end{lem}

\begin{proof}
First let us notice that for any $n$ and any  $V_{1i}, V_{2i}$, $1\leq i\leq n$, a set of i.i.d. random variables with $\mathbb{E}(V_{1i}^2+ V_{2i}^2)<\infty$ and
\[
A_n = \frac{1}{n(n-1)h} \sum_{1\leq i\neq j\leq n} V_{1i} V_{2j} K(h^{-1} \langle X_i - X_j ,\gamma^{(p)}_0 \rangle),
\]
there exists some constant $C$ (independent of $n$) such that
\begin{eqnarray}\label{rav_uu}
Var (A_n) & \leq & \frac{C}{n} Var( V_{1i} V_{2j} h^{-1}K(h^{-1} \langle X_i - X_j ,\gamma^{(p)}_0 \rangle) ) \notag \\
& \leq &  \frac{C}{nh^2} \mathbb{E}[ \zeta_1^2(\langle X_i  ,\gamma^{(p)}_0 \rangle) \zeta_2^2 (\langle X_j  ,\gamma^{(p)}_0 \rangle)  K^2(h^{-1} \langle X_i - X_j ,\gamma^{(p)}_0 \rangle) ]\notag\\
& \leq &  \frac{C}{nh^2} \mathbb{E}[ \zeta_1^2(\langle X_i  ,\gamma^{(p)}_0 \rangle) ]  \mathbb{E}[\zeta_2^2 (\langle X_j  ,\gamma^{(p)}_0 \rangle) ] = \frac{C}{nh^2} \mathbb{E}(V^2_{1i})\mathbb{E}(V^2_{2i})
\end{eqnarray}
where $\zeta_l^2(\langle X_i  ,\gamma^{(p)}_0 \rangle) = \mathbb{E}(V^2_{li}\mid \langle X_i  ,\gamma^{(p)}_0 \rangle) $, $l=1,2$. Since  $nh^2 \rightarrow \infty$, we have $Var (A_n)\rightarrow 0$.

Now, to  check $\widehat \tau_n^2 (\gamma_0^{(p)}) = \tau_n^2 (\gamma_0^{(p)}) \{ 1+o_{\mathbb{P}} (1)\}$ take $V_{1i}=V_{2i}=U_i^2$.
We have $\mathbb{E} [\widehat \tau_n^2 (\gamma_0^{(p)})\mid X_1,\cdots,X_n] = \tau_n^2 (\gamma_0^{(p)})$ and
\begin{equation}\label{mpz}
\mathbb{E}\{ \widehat \tau_n^2 (\gamma_0^{(p)}) -  \tau_n^2 (\gamma_0^{(p)})\}^2 = \mathbb{E}\{ Var[\widehat \tau_n^2 (\gamma_0^{(p)})\mid X_1,\cdots,X_n ]\} \leq Var (\widehat \tau_n^2 (\gamma_0^{(p)}))\rightarrow 0.
\end{equation}
By the fact that $Var(U\mid X^{(p)})$ is bounded and bounded away from zero almost surely, and the fact that for $l=2$ and $l=4$, $\mathbb{E}[h^{-1} K_{h}^l \left(\langle X_{1}-X_{2},\gamma\rangle \right)]$ is bounded and bounded away from zero $\forall p\geq 1$, $\forall \gamma\in B_p$ and $\forall h\leq 1$, deduce that the expectation of $\tau_n^2 (\gamma_0^{(p)})$ stays away from zero and infinity and its variance tends to zero. This together with (\ref{mpz}) allow to conclude that $\widehat \tau_n^2 (\gamma_0^{(p)}) = \tau_n^2 (\gamma_0^{(p)}) \{ 1+o_{\mathbb{P}} (1)\}$. To obtain the same conclusion with $\tau_n^2 (\gamma_0^{(p)})$ replaced by $v_n^2 (\gamma_0^{(p)})$ it suffices to consider above conditional expectations given
$\langle X_1  ,\gamma^{(p)}_0 \rangle,\cdots,\langle X_n  ,\gamma^{(p)}_0 \rangle.$
The arguments for $\widehat{v}_{n}^{2}$ are similar and hence will be omitted. \end{proof}

\bigskip

\begin{proof}[Proof of Lemma \ref{leem1}]
Let $M>0$ be a real number that depend on $n$ in a way that will be
specified later, define $\eta _{i}^{M}=U _{i}\mathbb{I}(
| U_{i}| \leq M ) -\mathbb{E}( U _{i}\mathbb{I}(| U _{i}| \leq
M) \mid X_{i}^{(p)}) $ and consider the degenerate $U$-process
\begin{equation*}
U_{n}\widetilde{g}=\frac{1}{n(n-1)}\sum\limits_{j\neq i}\eta _{i}^{M}\eta
_{j}^{M}K_{h}\left( \langle X_{i}-X_{j},\gamma\rangle \right) =\frac{1}{n(n-1)}%
\sum\limits_{j\neq i}\widetilde{g}((\eta _{i}^{M},X_{i}),(\eta
_{j}^{M},X_{j});h,\gamma )
\end{equation*}%
defined by the functions $\widetilde{g}$ indexed by $h$ and $\gamma\in\mathcal{S}^p $. By Assumption \ref{D}  and \ref{K}-(a),
the arguments used in Lemma \ref{entropy} above for the class $\mathcal{F}_{1p}$, and Lemma 9.9-(vi) of Kosorok (2008),
the bounded family $\mathcal{F}_{3p} = \{\widetilde{g}: \gamma  \in\mathcal{S}^p,h>0\}$ has a covering number like in (\ref{cov_entr}). By Theorem 2 of Major (2006) and its
corollary, where we  assume without loss of generality that $0\leq K(\cdot
)\leq 1$,
\begin{multline}
\hspace{-0.3cm}\mathbb{P}\!\left( \sup_{\gamma\in\mathcal{S}^p }\left\vert U_{n}\widetilde{g}\right\vert\! \geq \!
\frac{th^{1/2}\ln n p^{3/2}}{(n-1)}\right) \!= \!  \mathbb{P}\!\left( \sup_{\gamma\in\mathcal{S}^p }\left\vert
\frac{1}{n}\sum\limits_{j\neq i}\frac{\eta _{i}^{M}}{M}\frac{\eta _{j}^{M}}{M%
}K_{h}\left( \langle X_{i}\!-\!X_{j},\gamma \rangle \right) \right\vert\!\!
\geq \!   \frac{th^{1/2}p^{3/2}\ln n}{M^{2}}\!\right)  \notag \\
\leq  C_{1}C_{2}\exp \left\{ -C_{3}\left( \frac{th^{1/2} p^{3/2}\ln n}{M^{2}\sigma
_{M}}\right) \right\}, \qquad \mbox{\rm for any }t>0\;,
\end{multline}%
\begin{equation}
\mbox{\rm provided }\qquad n\sigma _{M}^{2}\geq \frac{th^{1/2} p^{3/2}\ln n}{%
M^{2}\sigma _{M}}\geq C_{4}\left[ p+\max \left( \ln C_{2}/\ln n,0\right) %
\right] ^{3/2}\ln \frac{2}{\sigma _{M}}  \label{cond_major}
\end{equation}%
where $C_{1},\ldots C_{4}>0$ are some constants independent on $n$, $h$ and $%
M$ and
\begin{equation*}
\sigma _{M}^{2}=\sup_{\gamma\in\mathcal{S}^p}\mathbb{E}\left[ \left( \frac{\eta _{i}^{M}}{M}\right)
^{2}\left( \frac{\eta _{j}^{M}}{M}\right) ^{2}K_{h}^{2}\left(\langle X_{i} -X_{j},\gamma \rangle\right) \right] \;.
\end{equation*}%
From Assumption \ref{D}-(b,c) and using the arguments as in the last part of the proof of Lemma \ref{tech_u_stat} above, there is a constant $C>0$ independent of $n$ such that $
C^{-1}\leq \sigma _{M}^{2}M^{4}/h\leq C$.  Take $M^{4}=nhp^{-3/2} \ln^{-(1+\delta )}
n \rightarrow \infty$ with $\delta >0$ arbitrarily small. Hence $\sigma
_{M}^{2}$ is of order $n^{-1} p^{3/2} \ln^{1+\delta} n \rightarrow 0$ and for any $t>0$
\begin{equation}
n\sigma _{M}^{2}\geq \frac{nh}{CM^{4}}=C^{-1} p^{3/2} \ln ^{1+\delta }n \geq
\frac{t h^{1/2} p^{3/2} \ln n}{M^{2}\sigma _{M}}  \label{aux1}
\end{equation}
provided $n$ is large enough. On the other hand, for any constant
$C^{\prime }>0$
\begin{equation}
\frac{th^{1/2}p^{3/2}\ln n }{M^{2}\sigma _{M}}\geq C^{-1/2}tp^{3/2} \ln n \geq C^{\prime }p^{3/2}\ln
n  \rightarrow \infty \label{aux2}
\end{equation}
for any sufficiently large $t$. Since $(\ln n )^{-1} \ln (2/\sigma_M)
$ is bounded by a positive constant as $n$ goes to $\infty$, Equations
(\ref{aux1}) and (\ref{aux2}) show that (\ref{cond_major}) is
satisfied for our $M$, with $n$ and $t$ large enough.  By
Theorem 2 of Major (2006), $U_{n}\widetilde{g}=O_{\mathbb{P}}\left(
n^{-1}h^{1/2}p^{3/2}\ln n\right) .$

Now, it remains to study the tails of $%
U _{i}$, that is we have to derive the orders of the remainder
terms
\begin{equation*}
2R_{1n}+R_{2n}=\frac{2}{n(n-1)}\sum\limits_{j\neq i}\eta _{i}^{M}\xi
_{j}K_{h}\left( \langle X_{i}-X_{j},\gamma\rangle \right) +\frac{1}{%
n(n-1)}\sum\limits_{j\neq i}\xi _{i}\xi _{j}K_{h}\left( \langle
X_{i}-X_{j},\gamma\rangle \right)
\end{equation*}%
where $\xi _{i}=U _{i}-\eta _{i}^{M}=U
_{i}\mathbb{I}\left( \left\vert U _{i}\right\vert >M\right) -%
\mathbb{E}\left[ U _{i}\mathbb{I}\left( \left\vert U
_{i}\right\vert >M\right) \mid X_{i}\right].$
Now, $\mathbb{E}\left[ \sup_{\gamma }\left\vert R_{1n}\right\vert \right]
\leq C\mathbb{E}\left( \left\vert \eta _{i}^{M}\right\vert \left\vert \xi
_{j}\right\vert \right) \leq 2C\mathbb{E}\left( \left\vert U
_{i}\right\vert \right) \mathbb{E}\left( \left\vert \xi _{j}\right\vert
\right) \leq C^{\prime }\mathbb{E}\left( \left\vert \xi _{j}\right\vert
\right) $,
and thus by H\"{o}lder's and Chebyshev's inequalities
\begin{equation*}
\mathbb{E}\left( \left\vert \xi _{i}\right\vert \right) \leq 2\mathbb{E}%
\left[ \left\vert U _{i}\right\vert \mathbb{I}\left( \left\vert
U _{i}\right\vert >M\right) \right] \leq 2\mathbb{E}^{1/m}\left[
\left\vert U _{i}\right\vert ^{m}\right] \mathbb{P}^{(m-1)/m}\left[
\left\vert U _{i}\right\vert >M\right] \leq 2\mathbb{E}\left[
\left\vert U _{i}\right\vert ^{m}\right] \,M^{1-m}.
\end{equation*}
Now it remains to choose $m$ sufficiently large such that $M^{1-m}=o\left( n^{-1}h^{1/2}p^{3/2}\ln n\right) $. With Assumption \ref{K}-(b) and our choice of $M$, $m > 11$ will be sufficient.
Also it is clear that $\sup_{\gamma }|R_{2n}|$ is of smaller order than $\sup_{\gamma }|R_{1n}|$.

To prove that the inverse of the variance estimate is bounded in probability, in view of Lemma \ref{rav_u}, it remains to show that $1/\tau_n^{2} (\gamma)$, $\gamma\in B_p$, is uniformly bounded in probability. For this recall that $\sigma_p^2 ( X^{(p)} ) \geq \underline{\sigma}^2$
and  apply Lemma \ref{tech_u_stat}. Now the proof is complete. \end{proof}

\bigskip

\begin{proof}[Proof of Lemma \ref{beta}]
By definition, $nh^{1/2}Q_{n}(\gamma
_{0}^{(p)})/ \widehat v_n (\gamma
_{0}^{(p)}) \leq nh^{1/2} Q_{n}(\widehat{\gamma }_{n}) / \widehat v_n (\widehat \gamma
_n) -\alpha
_{n} \mathbb{I} (\widehat{\gamma }_{n} \neq \gamma_{0}^{(p)})) $. This implies that
$$0 \leq \mathbb{I} (\widehat{\gamma }_{n} \neq \gamma_{0}^{(p)}) \leq nh^{1/2} \alpha
_{n}^{-1}\left\{ Q_{n}( \widehat{\gamma } _{n})/ \widehat v_n (\widehat \gamma
_n )- Q_{n}(\gamma _{0}^{(p)})  / \widehat v_n (\gamma
_{0}^{(p)})\right\} .$$ From Lemmas \ref{leem1}, \ref{tech_u_stat} and \ref{rav_u},
\begin{eqnarray*}
\left| \frac{Q_{n}(
\widehat{\gamma }_{n})}{ \widehat v_n (\widehat \gamma
_n)} - \frac{ Q_{n}(\gamma _{0}^{(p)}) }{ \widehat v_n (\gamma
_{0}^{(p)})} \right|&\leq &  2 \max \left[ \sup_{\gamma \in B_p} \{1/\widehat \tau_n^2 (\gamma)\} ,\;  1/\widehat v_n^2  \right]  \sup_{\gamma \in B_p}|Q_{n}(\gamma)| \\& = & O_{\mathbb{P}}(n^{-1}h^{-1/2}p^{3/2} \ln n) .
\end{eqnarray*}
Then $\alpha _{n} p^{-3/2} \!/\ln n
\!\!\rightarrow\!\! \infty$ yields $\mathbb{I}(\widehat{\gamma }_{n} \!\neq\! \gamma_{0}^{(p)})
\!=\! o_{\mathbb{P}}(1).$ Thus  $\mathbb{P}(\widehat{\gamma}_{n}\! \neq\! \gamma_{0}^{(p)})\! = \mathbb{E} [%
\mathbb{I}(\widehat{\gamma }_{n} \!\neq \!\gamma_{0}^{(p)})] \!\rightarrow \!0$.
\end{proof}

\bigskip

\begin{proof}[Proof of Theorem \ref{as_law}]
From Lemma \ref{beta}, the probabilities of the events  $\{ Q_{n}(
\widehat{\gamma }_{n}) = Q_{n}(\gamma _{0}^{(p)})\}$ and $
\{\widehat{v}_{n}^{2} (\widehat{\gamma}_{n}) = \widehat{\tau}
_{n}^{2} ({\gamma}_{0}^{(p)})\}$, with $\widehat{v}_{n}^{2} (\cdot)$  defined in (\ref{var_est1a}), both converge to 1.
On the other hand, by Lemma \ref{rav_u} above $\widehat \tau_n^2 (\gamma_0^{(p)}) = \tau_n^2 (\gamma_0^{(p)}) \{ 1+o_{\mathbb{P}} (1)\}$. Moreover,
$\widehat{v}_{n}^{2}  = \tau_{n}^{2} (\gamma_{0}^{(p)})\{1 + o_{\mathbb{P}} (1)\}$, with  $\widehat{v}_{n}^{2}$ defined in (\ref{var_est1b}), provided that condition (\ref{nonp}) holds true.
Hence it suffices to derive the asymptotic distribution of $nh^{1/2} Q_n(\gamma_0^{(p)})/\tau_n(\gamma_0^{(p)})$ under $H_0$.
For this purpose we use Assumption \ref{D}-(c)(iii) and proceed like in  Theorem 3.3 and Lemma 6.2 of Patilea and Lavergne (2008); see also the CLT in Lemma 2 of Guerre and Lavergne (2005). Moreover we use our Lemma \ref{tech_u_stat} with $k_1=k_2=0$ and $l=2.$
To be exactly in the case of Lavergne and Patilea (2008), first consider $nh^{1/2} Q_n(\gamma_0^{(p)})/v_n(\gamma_0^{(p)})$
 with $v_n(\gamma_0^{(p)})$ defined in (\ref{drev}).
The arguments for the asymptotic normality of $nh^{1/2} Q_n(\gamma_0^{(p)})/v_n(\gamma_0^{(p)})$ are identical to those of Lavergne and Patilea and hence will be omitted.
Finally, by Lemma \ref{rav_u},  $v_{n}^2 (\gamma_0^{(p)}) = \tau_n^2(\gamma_0^{(p)})\{1+o_{\mathbb{P}}(1)\}$ and the stated result follows.
\end{proof}

\bigskip

\begin{proof}[Proof of Theorem \ref{altern}]
The proof is based on inequality (\ref{eqaa}).
Since $\mathbb{E}(U^2\mid X) \geq \underline{\sigma}^2 + r_n^2\delta^2(X),$ $\mathbb{E}(U\mid X) =  r_n\delta(X),$ and $Var(U\mid \langle X,\gamma_0^{(p)}\rangle) \geq \underline{\sigma}^2 + r_n^2 Var(\delta(X)\mid \langle X,\gamma_0^{(p)}\rangle )$, clearly the variance estimate $\widehat v_n (\gamma^{(p)}_0))$ stays away from zero. Hence it suffices to look at the behavior of $Q_n(\gamma)$. By Lemma \ref{lem1}-(B) there exists $p_0$  and $\widetilde \gamma\in B_{p_0}\subset \mathcal{S}^{p_0}$ ($p_0$ and $\widetilde \gamma$ independent of $n$) such that $\mathbb{E}[\delta(X) \mid \langle X,\widetilde  \gamma\rangle]\neq 0$. Since $\max_{\gamma\in B_p}Q_n(\gamma) \geq Q_n(\widetilde \gamma)$ for any $p\geq p_0$, it suffices to investigate the rate of $Q_n(\widetilde \gamma)$. We can write
\begin{eqnarray*}
Q_n(\widetilde \gamma) & = & \frac{1}{n(n-1)h}\sum_{i\neq j} U_i^0 U_j^0 K_h(\langle X_i - X_j ,\widetilde \gamma\rangle )\\
&&+ \frac{2r_n}{n(n-1)h}\sum_{i\neq j} U_i^0 \delta(X_j) K_h(\langle X_i - X_j ,\widetilde \gamma\rangle )\\
&&+\frac{r_n^2}{n(n-1)h}\sum_{i\neq j} \delta(X_i) \delta(X_j) K_h(\langle X_i - X_j ,\widetilde \gamma\rangle )\\
&=:& Q_{0n}(\widetilde \gamma) + 2r_nQ_{1n}(\widetilde \gamma)+ r_n^2Q_{2n}(\widetilde \gamma).
\end{eqnarray*}
Since $\widetilde \gamma$ is fixed (and of finite dimension),  $Q_{0n}(\widetilde \gamma)=O_{\mathbb{P}}(n^{-1}h^{-1/2})$ (cf. proof of Theorem \ref{as_law}). The $U-$statistic  $Q_{1n}(\widetilde \gamma)$ can be decomposed in a degenerate $U-$statistic of order 2 with the rate $O_{\mathbb{P}}(h^{-1}n^{-1}) = O_{\mathbb{P}}(n^{-1/2})$ and the sum average of centered variables
$$
\frac{1}{n}\sum_{1\leq i\leq n} U_i^0 \mathbb{E}[\delta(X_j)h^{-1} K_h(\langle X_i - X_j ,\widetilde \gamma\rangle )\mid X_i].
$$
Hence it suffice to bound  $v_n^2 = \mathbb{E}\{(U_i^0)^2 \mathbb{E}^2[\delta(X_j) h^{-1}K_h(\langle X_i - X_j ,\widetilde \gamma\rangle )\mid X_i]\}$. There are several set of assumptions on  $\delta$ and $f_{\widetilde\gamma}$  that could be used. Condition (i) implies that the map $x\mapsto\mathbb{E}[h^{-1}K_h(\langle x - X_j ,\widetilde \gamma\rangle)]$ is bounded. This combined with the bounded conditional variance of $U_i^0$ and the finite second order moment of $\delta(X_j)$ yield $v_n^2 \leq c$ for some constant $c>0$.
Similar arguments could be combined with the condition (ii) to obtain the boundedness of $v_n^2.$ Finally, if condition (iii) is met, let $V_i =\langle X_i ,\widetilde \gamma\rangle $ and $\overline \delta (V_j) = \mathbb{E}[\delta(X_j)\mid V_j].$ Then using the inverse Fourier transform device we have
\begin{eqnarray*}
\mathbb{E}[\delta(X_j)h^{-1}K_h(V_i - V_j )\mid X_i] &=&  \mathbb{E}\left[\overline\delta (V_j) \int \exp\{i t (V_i - V_j) \} \mathcal{F}[K] (ht) dt \mid V_i\right]\\
&=&  \int_{\mathbb{R}} \exp\{i t V_i \} \mathcal{F}[\overline \delta f_{\widetilde\gamma} ](t)\mathcal{F}[K](ht) dt.
\end{eqnarray*}
Take absolute value in the last integral, use the fact that $\mathcal{F}[\overline \delta f_{\widetilde\gamma} ]\in L^1 (\mathbb{R})$, Lebesgue dominated convergence theorem and the fact that $\mathcal{F}[K] (ht)\rightarrow \mathcal{F}[K] (0)$ as $h\rightarrow 0$ to deduce that $\mathbb{E}[\delta(X_j)h^{-1}K_h(V_i - V_j )\mid X_i]$ is bounded, and so is $v_n^2$. Deduce that with any of the conditions (i) to (iii), $Q_{1n}(\widetilde \gamma)=O_{\mathbb{P}}(n^{-1/2}).$ Finally, it is easy to show that $Var[Q_{2n}(\widetilde \gamma)]\rightarrow 0$ (see, e.g.,the proof of equation (26) in Lavergne and Patilea (2008)). It remain to study
$$
\mathbb{E}[Q_{2n}(\widetilde \gamma)]=\int_{\mathbb{R}} |\mathcal{F}[\overline \delta f_{\widetilde\gamma} ]|^2(t)\mathcal{F}[K] (ht) dt.
$$
If condition (i) or (ii) holds true, $\overline \delta f_{\widetilde\gamma}\in L ^2 (\mathbb{R})$ and by Plancherel theorem and Lebesgue dominated convergence theorem, $\mathbb{E}[Q_{2n}(\widetilde \gamma)]\rightarrow \int_{\mathbb{R}} |\overline \delta f_{\widetilde\gamma}|^2>0$. If condition (iii) is met, $\mathcal{F}[\overline \delta f_{\widetilde\gamma} ]\in L ^2 (\mathbb{R})$ and since $\overline \delta f_{\widetilde\gamma} \in L ^1 (\mathbb{R})$, deduce that $\overline \delta f_{\widetilde\gamma} \in L ^2 (\mathbb{R})$ and continue with the same arguments. Deduce that with any of the conditions (i) to (iii), $ Q_{2n}(\widetilde \gamma) \asymp O_{\mathbb{P}}(1).$ Collecting the rates, we obtain the result.
\end{proof}

\subsection{Testing the functional linear model: proofs of the results}

To simplify notation, in this section we write $\|\cdot\|$ instead of $\|\cdot\|_{L^2}.$

%
%
%

\bigskip

\begin{proof}[Proof of Lemma \ref{diff_fund}]
By simple calculations, we have $\widehat{U}_i = U_i- \langle \widehat{b}-b,X_i-\overline{X}_n\rangle- \overline{U}_n $. Let $K_{h,ij}(\gamma)$ be a short notation for
$
K_{h}\left( \langle X_{i}-X_{j},\gamma\rangle  \right).$ We have the following decomposition
\begin{equation*}
Q_{n}(\gamma; \widehat a,\widehat b ) = Q_{n}(\gamma)-2V_1(\gamma)-2V_2(\gamma)+V_3(\gamma)+V_4(\gamma)+2V_5(\gamma)
\end{equation*}
where
\begin{equation*}
V_1 =  \frac{\overline{U}_n}{n(n-1) h}\sum_{1\leq i\neq j\leq n} U_i K_{h,ij}(\gamma),\;\; V_2 = \left\langle \widehat{b}-b,\frac{1}{n(n-1) h}\sum_{i\neq j} U_i( X_j -\overline{X}_n)K_{h,ij}(\gamma)\right\rangle,
\end{equation*}
\begin{equation*}
V_3 = \frac{1}{n(n-1) h}\sum_{ i\neq j}\langle \widehat{b}-b,X_i-\overline{X}_n\rangle \langle \widehat{b}-b,X_j-\overline{X}_n\rangle  K_{h,ij}(\gamma)
\end{equation*}
\begin{equation*}
V_4 = \frac{\overline{U}_n^2}{n(n-1) h}\sum_{ i\neq j} K_{h,ij}(\gamma), \;\; V_5 = \overline{U}_n\left\langle \widehat{b}-b,\frac{1}{n(n-1) h}\sum_{i\neq j} ( X_j -\overline{X}_n)K_{h,ij}(\gamma)\right\rangle.
\end{equation*}
To prove the rate in the first part of (\ref{eq_test_lin}) we will show that $$\sup_{\gamma\in \mathcal{S}^p} nh^{1/2}|V_j| =   o_{\mathbb{P}} (1),$$ for $j=1$ to $j=5$.
First let us notice that by Fubini Theorem, $\mathbb{E} (\| \overline{X}_n - \mathbb{E} (X) \|^2) = n^{-1}\int_0^1 Var(X(t))dt$ and so  $\| \overline{X}_n - \mathbb{E} (X) \|
= O_{\mathbb{P}}(n^{-1/2}).$

For $V_1$ use the fact that $\overline{U}_n = O_{\mathbb{P}}(n^{-1/2})$ and apply Lemma \ref{tech_emp_proc}. Thus there exists $a>0$ and   $ 0<\epsilon < 2a(1-2\zeta) $ such that  $$\sup_{\gamma\in \mathcal{S}^p} nh^{1/2}|V_1| = nh^{1/2}O_{\mathbb{P}}(n^{-1/2}) O_{\mathbb{P}}(n^{-1/2+\epsilon}p^{1/2}h^{-1/2 +a})= o_{\mathbb{P}}(1).$$
To derive the rate of  $V_2$ let us write
\begin{eqnarray*}
V_2 &=& \left\langle \widehat{b}-b, \frac{1}{n(n-1) h}\sum_{i\neq j} U_i\{ X_j -\mathbb{E} (X) \} K_{h,ij}(\gamma)\right\rangle\\
&&- \left\langle \widehat{b}-b, \overline{X}_n - \mathbb{E} (X)\right\rangle\frac{1}{n(n-1) h}\sum_{i\neq j} U_iK_{h,ij}(\gamma) \\
&=& V_{21} - V_{22}
\end{eqnarray*}
By Cauchy-Schwarz inequality, the rate of $\| \overline{X}_n - \mathbb{E} (X) \|$ and Lemma \ref{tech_emp_proc},
$\sup_{\gamma\in \mathcal{S}^p} |V_{22}| = o_{\mathbb{P}}(n^{-1/2}\|\widehat{b}-b\|)=o_{\mathbb{P}}(n^{-2\rho}) $.
For the rate of $V_{21}$ let us write
\begin{eqnarray*}
\frac{n-1}{n} V_{21}  &=&  h^{-1}\left[ \frac{1}{n}\sum_{1\leq  i\leq n} U_i K_{h,ij}(\gamma)\right] \frac{1}{n }\sum_{1\leq  j\leq n} \left\langle \widehat{b}-b,  \; X_j -\mathbb{E} (X) \right\rangle \\
&& - \left\langle \widehat{b}-b, \frac{K(0)}{n^2 h}\sum_{1\leq i\leq n} U_i\{ X_i -\mathbb{E} (X) \} \right\rangle\\
&=&  V_{211} - V_{212} .
\end{eqnarray*}
By Cauchy-Schwarz inequality and the law of large numbers with $|U_i|\| X_i -\mathbb{E} (X) \|,$  $V_{212}= n^{-1}h^{-1} O_{\mathbb{P}}(1) O_{\mathbb{P}}(\|\widehat b - b\|).$ Next, by Cauchy-Schwarz inequality we can write
$$
|V_{211}|\leq h^{-1} \left\{ \sup_{\gamma\in\mathcal{S}^p, t\in\mathbb{R}} |Z_n(\gamma,t)| \right\}  \| \widehat{b}-b\| \| \overline{X}_n - \mathbb{E} (X) \|,
$$
where
$$
Z_n(\gamma,t)= \frac{1}{n}\sum_{1\leq  i\leq n} U_i K_{h}(\langle X_i ,\gamma\rangle - t )
$$
Apply Lemma \ref{tech_emp_proc}  to deduce that there exists some small $\epsilon>0$ such that
$$
\sup_{\gamma\in \mathcal{S}^p}nh^{1/2}|V_{211}|= nh^{-1/2} O_{\mathbb{P}}(h^{1/2}n^{-1/2+\epsilon}p^{1/2})  O_{\mathbb{P}}( \| \widehat{b}-b\|) O_{\mathbb{P}}(n^{-1/2}) =o_{\mathbb{P}}(1).
$$
Deduce that
$$
\sup_{\gamma\in \mathcal{S}^p} nh^{1/2}| V_{2}| = o_{\mathbb{P}}(1).
$$
For $V_3$ take absolute values and use Cauchy-Schwarz inequality and triangle inequality:
$$
|V_3| \leq  \frac{\|\widehat{b}-b\|^2}{n(n-1) h}\sum_{i\neq j} \{\|X_i- \mathbb{E}(X_i)\| +  \| \overline{X}_n- \mathbb{E}(X)\| \}  \{ \|X_j- \mathbb{E}(X_j)\| +  \| \overline{X}_n- \mathbb{E}(X)\| \} K_{h,ij}(\gamma).
$$
Apply Lemma \ref{tech_u_stat} three times and deduce that
$$\sup_{\gamma\in \mathcal{S}^p} nh^{1/2}|V_3| = nh^{1/2}O_{\mathbb{P}}(\|\widehat{b}-b\|^2)O_{\mathbb{P}}(1)= o_{\mathbb{P}}(1).$$
For $V_4$ apply Lemma \ref{tech_u_stat} with $k_1 = 0$, $k_2=0$ and $l=1$ and the rate of $\overline{U}_n$ to deduce
$$\sup_{\gamma\in \mathcal{S}^p} nh^{1/2}|V_4| =  nh^{1/2}O_{\mathbb{P}}(n^{-1})O_{\mathbb{P}}(1)= o_{\mathbb{P}}(1).$$
Finally, let us write
\begin{eqnarray*}
V_5 &=& \frac{\overline{U}_n}{n(n-1) h}\sum_{i\neq j}\left\langle \widehat{b}-b,  X_j -\mathbb{E} (X) \right\rangle K_{h,ij}(\gamma)\\
&&-  \left\langle \widehat{b}-b,\overline{X}_n - \mathbb{E} (X)\right\rangle \frac{\overline{U}_n}{n(n-1) h}\sum_{i\neq j} K_{h,ij}(\gamma) \\
&=:& V_{51} + V_{52}.
\end{eqnarray*}
By Cauchy-Schwarz inequality and Lemma \ref{tech_u_stat} with $k_1 = 0$, $k_2=2$, $l=1$ and $U_j^2  $ replaced by $\|X_j- \mathbb{E}(X_j)\|$,
$$
\sup_{\gamma\in \mathcal{S}^p} nh^{1/2}| V_{51}| = nh^{1/2} O_{\mathbb{P}}(n^{-1/2}) O_{\mathbb{P}}(\| \widehat{b}-b\| ) O_{\mathbb{P}}(1)= n^{1/2}h^{1/2}O_{\mathbb{P}}(\| \widehat{b}-b\|)= o_{\mathbb{P}}(1).
$$
Next, similar arguments for the uniform rate of $V_{52}$.
Deduce that
$$
\sup_{\gamma\in \mathcal{S}^p} nh^{1/2}| V_{5}| = o_{\mathbb{P}}(1).
$$
The arguments for the rate in the second  part of (\ref{eq_test_lin}) are similar and hence will be omitted.
\end{proof}

\bigskip

\begin{proof}[Proof of Lemma \ref{leem_order_3}]
Let $\widehat{g}$ (resp. $\widehat{g}^0$) be the random function defined in (\ref{g_funct_i}) that one would obtain under the null (resp. alternative) hypothesis, that is with covariates $X_i$ and responses $a + \langle b, X_i \rangle   + U^0_i $ (resp. $a + \langle b, X_i \rangle + \delta(X_i)  + U^0_i $).  We can write
\begin{eqnarray*}
\|\widehat{b}^0-\widehat{b}\|^2
&=&\sum_{j=1}^m (\widehat{b}^0_j-\widehat{b}_j)^2  =  \sum_{j=1}^m \widehat{\theta}^{-2}_j \vert \langle \widehat{g}-\widehat{g}^0, \widehat{\phi}_j \rangle \vert^2
 \leq  \sum_{j=1}^m \widehat{\theta}^{-2}_j \Vert  \widehat{g}-\widehat{g}^0 \Vert^2 \Vert \widehat{\phi}_j \Vert^2\\
& = & r_n^2  \; \int_0^1 \!\!\left( \frac{1}{n}\sum_{i=1}^n \delta(X_i)\{X_i(u)-\overline{X}_n(u) \}\right)^2 \!\!\!du \;\;\sum_{j=1}^m \widehat{\theta}^{-2}_j
=: r_n^2 \Gamma_n \sum_{j=1}^m \widehat{\theta}^{-2}_j .
\end{eqnarray*}
We have
\begin{eqnarray*}
\esp \! \int_0^1 \!\! \left( \frac{1}{n} \sum_{i=1}^n \delta(X_i)\{X_i(u)-\esp X_i(u)\}\right)^2\!\!du & = &\!\! \int_0^1 \esp\left( \frac{1}{n} \sum_{i=1}^n \delta(X_i)\{X_i(u)-\esp X_i(u)\}\right)^2\!\!du\\
& = & \frac{1}{n^2} \int_0^1 \sum_{i=1}^n \esp[\delta^2(X_i)\{X_i(u)-\esp X(u)\}^2] du\\
& = & \frac{1}{n} \int_0^1 \esp[\delta^2(X_1)\{X_1(u)-\esp X_1(u)\}^2] du\\
& \leq & \frac{1}{n} \esp^{1/2} [\delta^4(X)] \esp [\Vert X-\esp X \Vert^2],
\end{eqnarray*}
where for the second equality we used the fact that $\mathbb{E}[\delta (X)\{X -\mathbb{E}X \}]=0 .$ On the other hand, since $\mathbb{E}[\delta (X)]=0 $, by the law of large numbers $\overline{\delta(X)}_n =  n^{-1}\sum_{i=1}^n \delta (X_i) = o_{\mathbb{P}}(1).$ Recall that  $\| \overline{X}_n - \mathbb{E} (X) \|
= O_{\mathbb{P}}(n^{-1/2}).$ Deduce that
$$\int_0^1 \left( \frac{1}{n} \sum_{i=1}^n \delta(X_i)\{\overline{X_n}(u)-\esp X(u)\}\right)^2\textrm{d}u = \overline{\delta(X)}_n^2 \| \overline{X}_n - \mathbb{E} (X) \| ^2 = o_{\mathbb{P}}(n^{-1}),$$
and finally that $\Gamma_n = O_{\mathbb{P}}(n^{-1}) .$ For the last part, use Theorem 1 of Hall and Horowitz (2007) which provides the rate of  $\int_0^1 \{\widehat{b}^0(u)-b(u)\}^2 \textrm{d}u$. Next, let us recall that Assumption \ref{order}-(c) implies $\theta_j \geq c j^{-\alpha}$ for some constant $c$ and thus $\sum_{j=1}^m \theta^{-2}_j = O (n^{(2\alpha +1)/( \alpha + 2\beta )}) = o(n)$  provided that  $m \asymp n^{1 / ( \alpha + 2\beta )}.$
Finally, one can  deduce from the equations (5.6) to (5.9) of Hall and Horowitz (2007) that $\sum_{j=1}^m (\theta^{-2}_j - \widehat{\theta}^{-2}_j ) = o_{\mathbb{P}}(n)$. Now the proof is complete.
\end{proof}

\bigskip
\begin{proof}[Proof of Theorem \ref{altern_b}]
Like in the proof of Theorem \ref{altern}, it suffice to show that $Q_n(\widetilde \gamma)\asymp_{\mathbb{P}}r_n^2$  for some fixed $p$ sufficiently large and $\widetilde \gamma\in B_p$, where
$$
Q_n(\widetilde \gamma)  =  \frac{1}{n(n-1)h}\sum_{i\neq j} \widehat U_i \widehat U_j K_h (\langle X_i - X_j , \widetilde \gamma\rangle ).
$$
Here $\widehat U_i$ are defined as in (\ref{eq_alty}). There are 15 cross-product terms, all of them similar or identical to those analyzed in the proofs of Theorem \ref{altern} and Lemma \ref{diff_fund}. For the sake of brevity we omit the details.
\end{proof}

\bigskip

\bigskip

\begin{center}
{\small REFERENCES }
\end{center}

{\small \parindent=0cm }

\newcounter{refs}
\begin{list}{}{\usecounter{refs}
\setlength{\labelwidth}{0in}
\setlength{\labelsep}{0in}\setlength{\leftmargin}{0in}
\setlength{\rightmargin}{0in} \setlength{\topsep}{0in}
\setlength{\partopsep}{0in} \setlength{\itemsep}{0cm} }

\item {\footnotesize \textsc{Bierens, H.J.} (1990). A consistent conditional
moment test of functional form. \textsl{Econometrica} \textbf{58},
1443--1458. }

\item {\footnotesize \textsc{Cai, T.,  and Hall, P.}  (2006). Prediction in functional linear regression.  \textsl{%
Annals of Statistics} \textbf{34}, 2159--2179. }

\item {\footnotesize \textsc{Cardot, H., Ferraty, F., Mas, A., and Sarda, P.}  (2003). Testing Hypotheses in the Functional
Linear Model. \textsl{Scandinavian Journal of Statistics} \textbf{30}, 241--255.}

\item {\footnotesize \textsc{Cardot, H., Goia, P., and Sarda, P.} (2004). Testing for no effect in functional linear regression models, some computational approaches. \textsl{Communications in Statistics - Simulation and Computation} \textbf{33}, 179--199.  }

\item {\footnotesize \textsc{Chen, D., Hall, P., M\"{u}ller, H.G.} (2011). Single and multiple index functional regression models with nonparametric link
\textsl{%
Annals of Statistics} \textbf{39}, 1720--1747}

\item {\footnotesize \textsc{Chow, Y.S., and Teicher, H.} (1997). \textsl{Probability Theory: Independence, Interchangeability, Martingales} (3rd ed.) Springer-Verlag, New-York. }

\item {\footnotesize \textsc{Crambes, C., Kneip, A., and Sarda, P.} (2008). Smoothing splines estimators for functional linear
regression. \textsl{Annals of Statistics} \textbf{37}, 35--72.}

\item {\footnotesize \textsc{Delsol, L., Ferraty, F., and Vieu, P.}
(2011). Structural test in regression on functional variables.
\textsl{Journal of Multivariate Analysis} \textbf{102},
422--447.}

\item {\footnotesize \textsc{Ferraty, F.} (Ed.) (2011). \textsl{Recent Advances in Functional Data Analysis and Related Topics.} Springer-Verlag Berlin Heidelberg.}

\item {\footnotesize \textsc{Ferraty, F., and Vieu, P.} (2006).
     \textsl{Nonparametric Functional Data Analysis: Theory and Practice}.
     Springer,  Berlin.}


\item {\footnotesize \textsc{Guerre, E., and Lavergne, P.} (2005).
Data-driven rate-optimal specification testing in regression models. \textsl{%
Annals of Statistics} \textbf{33}, 840--870. }

\item {\footnotesize \textsc{Hall, P., and Horowitz, J.L. } (2007).
Methodology and convergence rates for functional linear regression. \textsl{Annals of
Statistics} \textbf{35}, 70--91. }

\item {\footnotesize \textsc{H\"{a}rdle, W., and Mammen, E.} (1993).
Comparing nonparametric versus parametric regression fits. \textsl{Annals of
Statistics} \textbf{21}, 1296--1947. }

\item {\footnotesize \textsc{Horowitz, J.L., and Spokoiny, V.G.} (2001). An
adaptive, rate-optimal test of a parametric model against a nonparametric
alternative. \textsl{Econometrica} \textbf{69}, 599--631. }

\item {\footnotesize \textsc{Horv\`{a}th, L.,  and Reeder, R.} (2011).
A test of significance in functional quadratic regression.
arXiv:1105.0014v1 [math.ST].}

\item {\footnotesize \textsc{Kosorok, M.R.} (2008). \textsl{Introduction to Empirical Processes and Semiparametric Inference.} Springer Series in Statistics. Springer-Verlag, New-York.}

\item {\footnotesize \textsc{Lavergne, P. and Patilea, V.} (2008). Breaking the curse of dimensionality in nonparametric testing. \textsl{Journal of Econometrics} \textbf{143}, 103--122. }


\item {\footnotesize \textsc{Major, P.} (2006). An estimate on the
supremum of a nice class of stochastic integrals and U-statistics.
\textsl{Probability Theory and Related Fields}
\textbf{134}, 489--537.}

\item {\footnotesize \textsc{Mammen, E.} (1993). Bootstrap and Wild Bootstrap for High Dimensional Linear Models. \textsl{Annals of Statistics} \textbf{21}, 255--285.}

\item {\footnotesize \textsc{M\"{u}ller, H.G. and Stadtm\"{u}ller, U.} (2005). Generalized
functional linear models. \textsl{Annals of Statistics} \textbf{33},
774--805. }

\item {\footnotesize \textsc{Nolan, D., and Pollard, D.} (1987). $U-$%
processes : Rates of convergence. \textsl{Annals of  Statistics} \textbf{15},
780--799. }

\item {\footnotesize \textsc{Parthasarathy, K.R.} (1967). \textsl{Probability measures
 on metric spaces.} A.M.S. New-York. }

\item {\footnotesize \textsc{Ramsay, J., and Silverman, B.W.} (2005).
     \textsl{Functional Data Analysis} (2nd ed.).
     Sprin\-ger-Verlag, New York.}

\item {\footnotesize \textsc{Rudin, W.} (1987). \textsl{Real and
complex analysis.} McGraw-Hill. }

\item {\footnotesize \textsc{Stute, W.} (1997). Nonparametric models checks
for regression. \textsl{Annals of Statistics} \textbf{25}, 613--641. }

\item {\footnotesize
 \textsc{van der Vaart, A.D.,  and Wellner, J.A.} (1996). \textsl{Weak convergence and empirical processes.} Springer Series in Statistics. Springer-Verlag, New-York.}

\item {\footnotesize \textsc{Yao, F., and  M\"{u}ller, H.G.} (2010). Functional quadratic regression. \textsl{Biometrika} \textbf{97}, 49--64}.

\end{list}

\end{document}